\documentclass[12pt]{amsart}
\usepackage{txfonts}
\usepackage{amsmath,amssymb,amsfonts,amscd,mathtools}
\usepackage{eucal,mathrsfs}
\usepackage[all]{xy}
\usepackage{xspace}
\usepackage{graphicx}
\usepackage{float}
\usepackage{color}
\usepackage{fancybox}
\usepackage{multicol}
\usepackage{tikz}
\usepackage{enumitem}
\usepackage{microtype}
\usepackage[colorlinks,final,backref=page,hyperindex]{hyperref}

\newtheorem{theorem}{Theorem}[section]
\newtheorem{proposition}[theorem]{Proposition}
\newtheorem{lemma}[theorem]{Lemma}
\newtheorem{corollary}[theorem]{Corollary}

\theoremstyle{definition}
\newtheorem{definition}[theorem]{Definition}
\newtheorem{example}[theorem]{Example}

\newtheorem{remark}[theorem]{Remark}

\newcommand{\kfield}{\mathbf{k}}
\newcommand{\U}{\mathcal{U}}
\newcommand{\Hom}{\operatorname{Hom}}
\newcommand{\Ext}{\operatorname{Ext}}
\newcommand{\Tor}{\operatorname{Tor}}
\newcommand{\im}{\operatorname{Im}}
\newcommand{\id}{\operatorname{id}}

\allowdisplaybreaks
\begin{document}

\title[Modules and operator rings of Reynolds algebras]{Modules and Operator Rings of Reynolds Algebras}

\author{Jun He}
\address{School of Mathematics and Statistics, Nanjing University of Information Science \& Technology, Nanjing, Jiangsu 210044, P. R. China}
\email{he\_junn@163.com}

\author{Xiao-Song Peng}
\address{School of Mathematics and Statistics, Jiangsu Normal University, Xuzhou, Jiangsu 221116, P. R. China}
\email{pengxiaosong3@163.com}

\author{Chenghao Zhang}
\address{School of Mathematics and Statistics, Nanjing University of Information Science \& Technology, Nanjing, Jiangsu 210044, P. R. China}
\email{zchenghao@nuist.edu.cn}

\author{Yi Zhang}
\address{School of Mathematics and Statistics, Nanjing University of Information Science \& Technology, Nanjing, Jiangsu 210044, P. R. China}
\email{zhangy2016@nuist.edu.cn}

\date{}

\hypersetup{
  linkcolor=red,
  citecolor=green,
  urlcolor=cyan,
  pdftitle={Modules and Operator Rings of Reynolds Algebras},
  pdfauthor={Jun He; Xiao-song Peng; Chenghao Zhang; Yi Zhang},
  pdfsubject={Modules and operator rings of Reynolds algebras},
  pdfkeywords={Reynolds algebra, Reynolds module, ring of Reynolds operators, tensor product, Morita equivalence, differential module, derived functor}
}

\begin{abstract}
We develop a module theory for weighted Reynolds algebras by constructing
an associative ring that records the algebra action and the module
operator. The categories of left and right Reynolds modules are equivalent
to the corresponding module categories over this ring, and their tensor
products agree under these equivalences. We construct free Reynolds
modules and give an explicit graded description of the operator ring for
idempotent classical Reynolds operators fixing the identity. When the
Reynolds operator on the algebra is bijective, modules with bijective
module operators correspond to differential modules satisfying an
invertibility condition.
The operator ring transfers the Hom-tensor adjunction, Morita theory and
derived functors to Reynolds modules. Explicit computations over a field
show that the identity operator gives an operator ring of global dimension
one, whereas the zero operator gives one of infinite global dimension.
\end{abstract}

\subjclass[2020]{16W99, 16S10, 16D40, 13N15, 18G10}
\keywords{Reynolds algebra, Reynolds module, ring of Reynolds operators, tensor product, Morita equivalence, differential module, derived functor}

\maketitle

\setcounter{tocdepth}{2}
\tableofcontents

\vspace{-1cm}

\setcounter{section}{0}

\section{Introduction}\label{sec:introduction}

The module theory of a Reynolds algebra depends on both the algebra
action and the Reynolds operator. To describe this dependence, we seek an
associative ring whose modules recover Reynolds modules and whose
structure reflects the operator. We construct this ring for every weighted
Reynolds algebra and determine it explicitly for idempotent classical
Reynolds operators fixing the identity. Its module categories provide free
Reynolds modules, tensor products and derived functors.

\subsection{Reynolds operators}\label{subsec:intro-operators}

Reynolds operators arose from averaging procedures in the statistical
theory of turbulent flow \cite{Reynolds1895}. Their early algebraic theory
was developed in \cite{Dubreil1957,BillikRota1960,Birkhoff1961,Rota1964},
and averaging and Reynolds operators on Banach algebras were considered
in \cite{Miller1966}. In invariant theory, Reynolds operators occur as
projections onto invariant subalgebras
\cite{MumfordFogartyKirwan1994,DerksenKemper2015}.

Let $\kfield$ be a commutative ring with identity and let $R$ be an
associative unital $\kfield$-algebra. A classical Reynolds operator on $R$ is a linear map
$P:R\to R$ satisfying
\begin{equation}\label{eq:intro-reynolds}
 P(x)P(y)
 =
 P(xP(y))+P(P(x)y)-P(P(x)P(y))
\end{equation}
for all $x,y\in R$. Finite-group averaging and normalized integration
provide standard examples.

The weighted version considered in this paper replaces the last term in
Eq.~(\ref{eq:intro-reynolds}) by
$\lambda P(P(x)P(y))$, where $\lambda\in\kfield$. Thus
$P$ satisfies
$$
 P(x)P(y)
 =
 P(xP(y))+P(P(x)y)+\lambda P(P(x)P(y)).
$$
For $\lambda=-1$, this is the classical Reynolds identity
in Eq.~(\ref{eq:intro-reynolds}), while for $\lambda=0$ it becomes the
Rota-Baxter identity of weight zero. Weighted Reynolds identities and
free Reynolds algebras were considered in \cite{ZhangGaoGuo2021}.
Free commutative Reynolds algebras were subsequently constructed by
Gr\"obner-Shirshov methods in \cite{ZhangGaoGuo2023}. Generalized
Reynolds identities arising from Volterra integral operators have
also been considered in \cite{GuoGustavsonLi2024}.

The specialization at weight zero makes Rota-Baxter algebra theory a natural
point of comparison. Rota-Baxter algebras have a well-developed structure
theory \cite{Baxter1960,Atkinson1963,Guo2012}, and their modules and the
corresponding ring of Rota-Baxter operators have been studied in
\cite{QiaoGaoGuo2019,GuoLin2021}. However, the weighted Reynolds identity
contains the term $P(P(x)P(y))$, and this leads to module identities and
operator-ring relations different from those for Rota-Baxter algebras.

Reynolds operators have also been considered on Lie algebras and
$n$-Lie algebras \cite{Das2021Lie,HouSheng2023}, and in connection with
relative cocycle weighted operators and related structures on pre-Lie
algebras \cite{GuoZhang2023}. The present paper concerns modules over
associative Reynolds algebras.

We also record a structural feature of idempotent Reynolds operators.
For the classical weight $-1$, idempotent Reynolds operators are precisely
idempotent averaging operators \cite{PeiGuo2015}. More generally, we show that an idempotent
linear map $P$ is a Reynolds operator of weight $\lambda$ if and only if
it is an averaging operator and
$(1+\lambda)\operatorname{Im}P=0$. Consequently, if the base ring is a
field, a nonzero idempotent Reynolds operator can occur only at weight
$-1$.

\subsection{Reynolds modules}\label{subsec:intro-modules}

For an algebra with a linear operator, a module carries a linear operator
compatible with the algebra action. In the Reynolds case, this
compatibility is expressed by the Reynolds module identity introduced
below. The corresponding theories for
Rota-Baxter and differential algebras provide useful points of comparison:
free, projective, injective and flat Rota-Baxter modules were studied in
\cite{QiaoGaoGuo2019}, the description of Rota-Baxter modules by an associative ring
was developed in \cite{GuoLin2021}, and the representation theory
was recently extended to multiple Rota-Baxter algebras and modules in
\cite{HPZ}. Differential modules and their derived functors were considered
in \cite{QiaoPei2026}.

Let $(R,P)$ and $(S,T)$ be weighted Reynolds algebras. We introduce
left and right $(R,P)$-modules and
$(S,T)$-$(R,P)$-bimodules. For right modules, we use the convention
compatible with right actions of the same operator ring that describes
left modules. The right module identity induced by the opposite algebra
is related to this convention by a sign change in the module operator.
We also extend the Reynolds double product to left modules.

A first problem is the construction of free Reynolds modules. Starting from
free operated $R$-modules, we construct free left $(R,P)$-modules by
quotienting out the submodule generated by the Reynolds module relations.
We also show that ordinary free $R$-modules equipped with the coefficientwise
operator satisfy a restricted universal property.

The main structural tool of the paper is an associative ring that encodes
the module operator. Let $\kfield[Q]$ be the polynomial algebra in one
indeterminate $Q$. For a Reynolds algebra $(R,P)$ of weight $\lambda$, we
define the \emph{ring of Reynolds operators} $\U_{R,P}$ as the quotient of
the free product $R*_{\kfield}\kfield[Q]$ by the relations
$$
 P(r)Q
 =
 QP(r)+Q(r+\lambda P(r))Q,
 \quad r\in R.
$$
\begin{theorem}[= Proposition~\ref{prop:enveloping-universal} and
Theorems~\ref{thm:category-equivalence} and~\ref{thm:tensor-enveloping}]
\label{thm:A}
For every Reynolds algebra $(R,P)$, the canonical map
$R\to\U_{R,P}$ is injective, so $R$ may be identified with a
subalgebra of $\U_{R,P}$. The category of left
$(R,P)$-modules is naturally equivalent to the category of left
$\U_{R,P}$-modules, and the analogous statement holds for right modules.
Moreover, $(S,T)$-$(R,P)$-bimodules correspond to
$\U_{S,T}$-$\U_{R,P}$-bimodules. For every right $(R,P)$-module $M$
and left $(R,P)$-module $N$, there is a natural isomorphism
$$
 M\otimes_{(R,P)}N
 \cong
 M\otimes_{\U_{R,P}}N.
$$
\end{theorem}

Theorem~\ref{thm:A} reduces the categorical constructions for Reynolds modules
to those for modules over $\U_{R,P}$. In particular, the free Reynolds
module constructed from operated modules is identified with a free
$\U_{R,P}$-module.

For idempotent classical Reynolds operators fixing the identity, we can
describe the ring of Reynolds operators explicitly. Suppose that $P$ has weight $-1$,
$P^2=P$ and $P(1_R)=1_R$, and put $A=\operatorname{Im}P$. Then
$P:R\to A$ is an $A$-bimodule homomorphism, and the defining relations of
$\U_{R,P}$ simplify substantially.

\begin{theorem}[= Theorem~\ref{thm:idempotent-normal-form}]\label{thm:B}
Under the above assumptions, there is a graded algebra isomorphism
$$
 \U_{R,P}
 \cong
 R\oplus\bigoplus_{n\geq1}(R\otimes_A R)t^n.
$$
Here $(R\otimes_A R)t^n$ denotes a copy of $R\otimes_A R$ in
degree $n$, and $(r\otimes s)t^n$ corresponds to $r\otimes_A s$.
The degree-zero part has the original multiplication of $R$,
and $R$ acts on the positive-degree components through the two tensor
factors. For $r,s,u,v\in R$ and $i,j\geq1$,
$$
 ((r\otimes s)t^i)((u\otimes v)t^j)
 =
 (rP(su)\otimes v)t^{i+j}.
$$
\end{theorem}

For example, let a finite group $G$ act on $R$ by algebra automorphisms
and assume that $|G|$ is invertible in $\kfield$. For the averaging
operator associated with this action, the positive-degree components are
copies of $R\otimes_{R^G}R$, where $R^G$ is the subalgebra of
$G$-invariant elements.

The relation between bijective classical Reynolds operators and
derivations is recalled in \cite{ZhangGaoGuo2021}. For a bijective
Reynolds operator $P$ of weight $\lambda$, the map
$\delta=P^{-1}+\lambda\id_R$ is a derivation. We establish the
following correspondence for modules.

\begin{theorem}[= Theorem~\ref{thm:invertible-equivalence}]\label{thm:C}
Let $(R,P)$ be a Reynolds algebra of weight $\lambda$ with $P$
bijective, and let $\delta=P^{-1}+\lambda\id_R$. The assignment
$$
 (M,p_M)
 \longmapsto
 (M,p_M^{-1}+\lambda\id_M)
$$
defines an equivalence between the full subcategory of left $(R,P)$-modules
with bijective module operator and the full subcategory of left
differential $(R,\delta)$-modules $(M,d)$ for which
$d-\lambda\id_M$ is bijective.
\end{theorem}

The equivalence in Theorem~\ref{thm:A} also transfers standard categorical
constructions to Reynolds modules. We define tensor products
over a Reynolds algebra and identify them with tensor products over
$\U_{R,P}$. This yields the Hom-tensor adjunction and Reynolds versions
of the Eilenberg-Watts theorem and Morita theory. As a concrete
application, if $P$ has weight $-1$ and $P(1_R)=1_R$, then for the
entrywise Reynolds operator $P_n$ on $M_n(R)$ we obtain
$$
 \U_{M_n(R),P_n}
 \cong
 M_n(\U_{R,P}),
$$
and hence the corresponding Reynolds module categories are Morita
equivalent.

Finally, we obtain the homological constructions for Reynolds modules
from the corresponding constructions over $\U_{R,P}$.
Projective, injective and flat Reynolds modules correspond precisely to the
respective classes of modules over $\U_{R,P}$. We define
$\Ext^n_{(R,P)}$ and $\Tor_n^{(R,P)}$ and identify them with the ordinary
derived functors over $\U_{R,P}$. For Reynolds bimodules, the Hom-tensor
formalism gives Grothendieck spectral sequences for $\Ext$ and $\Tor$,
together with derived versions of tensor associativity and the Hom-tensor
adjunction.

Even for a fixed underlying algebra, the homological behavior can depend
strongly on the Reynolds operator. Over a field $\kfield$, the Reynolds
algebras $(\kfield,\id_{\kfield})$ and $(\kfield,0)$ of weight $-1$
have rings of
Reynolds operators
$$
 \U_{\kfield,\id}\cong\kfield[Q]
 \quad\text{and}\quad
 \U_{\kfield,0}\cong\kfield[Q]/(Q^2),
$$
respectively. The former has global dimension $1$, whereas the latter has
infinite global dimension, as computed in
Proposition~\ref{prop:one-dimensional-derived}.

\medskip
\noindent\textbf{Organization of this paper.}
Section~\ref{sec:modules} introduces Reynolds modules and their basic
constructions, including the correspondence with differential modules.
Section~\ref{sec:free} constructs free Reynolds modules and considers a
restricted universal property for ordinary free modules.
Section~\ref{sec:enveloping} constructs the ring of Reynolds operators,
identifies its module categories and proves the graded description for
idempotent operators fixing the identity.
Section~\ref{sec:tensor} develops tensor products, the Hom-tensor
adjunction and Morita theory, with an application to matrix algebras.
Section~\ref{sec:homological} treats projective, injective and flat
modules and their derived functors, and concludes with explicit
computations over a field.

\medskip
\noindent\textbf{Notation.}
Throughout the paper, $\kfield$ denotes a commutative ring with identity.
Unless otherwise specified, a module means a $\kfield$-module, linear and
bilinear maps are $\kfield$-linear and $\kfield$-bilinear, respectively,
and a linear functor is $\kfield$-linear. Tensor products without a
subscript are taken over $\kfield$. All algebras are associative unital
$\kfield$-algebras, and modules over algebras are unital unless stated
otherwise. A subalgebra need not contain the identity of the ambient
algebra. The unit of an algebra $R$ is denoted by $1_R$, and $\id_M$
denotes the identity map of a module $M$.

\section{Reynolds algebras and Reynolds modules}\label{sec:modules}

In this section, we first recall weighted Reynolds algebras.
We then introduce the concepts of Reynolds modules and bimodules,
give their basic constructions and relate modules with bijective
operators to differential modules.

\subsection{Weighted Reynolds algebras}\label{subsec:weighted}

In this subsection, we recall weighted Reynolds algebras and then study idempotent Reynolds operators.

\begin{definition}[{\cite{ZhangGaoGuo2021}}]\label{def:reynolds-algebra}
Let $\lambda\in\kfield$. A \textbf{Reynolds algebra of weight $\lambda$} is a pair $(R,P)$ consisting of an algebra $R$ and a linear map $P:R\to R$ such that, for all $x,y \in R$,
\begin{equation}\label{eq:weighted-reynolds}
 P(x)P(y)=P(xP(y))+P(P(x)y)+\lambda P(P(x)P(y)).
\end{equation}

\end{definition}

At weight $-1$, Eq.~(\ref{eq:weighted-reynolds}) reduces to
Eq.~(\ref{eq:intro-reynolds}).

\begin{example}\label{ex:normalized-integral} \cite{Birkhoff1961,Rota1964}
Let $\kfield=\mathbb R$ or $\mathbb C$, let $(X,\Sigma,\mu)$ be a measure space with $0<\mu(X)<\infty$, and give $R=L^\infty(X,\mu)$ its pointwise product. Define
$$
 P(f)=\frac{1}{\mu(X)}\int_X f\,d\mu\,\mathbf 1.
$$
Here $\mathbf 1$ is the constant function with value $1$.
Then $P$ is an idempotent Reynolds operator of weight $-1$.
\end{example}

\begin{example}\label{ex:hardy-integral} \cite{ZhangGaoGuo2021}
Let $b>0$ and $R=C([0,b],\mathbb R)$. Define
$$
 P(f)(x)=
 \begin{cases}
 \displaystyle\frac{1}{x}\int_0^x f(t)\,dt,&x>0,\\[6pt]
 f(0),&x=0.
 \end{cases}
$$
Then $P(f)$ is continuous and $P$ is a Reynolds operator of weight $-1$.
\end{example}

The following construction provides a standard way to obtain new Reynolds algebras from a given one.

\begin{example}\label{ex:tensor-reynolds-algebra}
Let $(R,P)$ be a Reynolds algebra of weight $\lambda$, and let $B$ be an
algebra. Equip $R\otimes B$ with the algebra structure determined by
$(r\otimes b)(s\otimes c)=rs\otimes bc$ for $r,s\in R$ and $b,c\in B$.
Define $P_B:=P\otimes\id_B:R\otimes B\longrightarrow R\otimes B$. Then $(R\otimes B,P_B)$ is a Reynolds algebra of weight $\lambda$.
Indeed, for $x=r\otimes b$ and $y=s\otimes c$,
\begin{align*}
 & P_B(x)P_B(y)
 = (P(r)\otimes b)(P(s)\otimes c)
  =P(r)P(s)\otimes bc\\
=& \bigl(P(rP(s))+P(P(r)s)
    +\lambda P(P(r)P(s))\bigr)\otimes bc\\
=& P_B\bigl(xP_B(y)\bigr)
   +P_B\bigl(P_B(x)y\bigr)
   +\lambda P_B\bigl(P_B(x)P_B(y)\bigr).
\end{align*}
Since both sides are bilinear in $x$ and $y$, the identity holds for all
$x,y\in R\otimes B$. Two special cases are as follows.

\begin{enumerate}[label=\textup{(\alph*)}]
\item For $B=\kfield[t]$, where $t$ is a central indeterminate, the
identification $R\otimes\kfield[t]\cong R[t]$ gives the coefficientwise
Reynolds operator
$$
 \widetilde P\left(\sum_{i=0}^{d}r_it^i\right)
 =\sum_{i=0}^{d}P(r_i)t^i.
$$
Thus $(R[t],\widetilde P)$ is a Reynolds algebra of weight $\lambda$.

\item For $B=M_d(\kfield)$, the identification
$R\otimes M_d(\kfield)\cong M_d(R)$ gives the entrywise operator
$$
 P_d:M_d(R)\longrightarrow M_d(R),
 \quad P_d((r_{ij}))=(P(r_{ij})).
$$
Hence $(M_d(R),P_d)$ is a Reynolds algebra of weight $\lambda$.
\end{enumerate}
\end{example}

\begin{example}\label{ex:group-averaging-algebra} \cite{DerksenKemper2015,MumfordFogartyKirwan1994,Maithani2026}
Let a finite group $G$ act on an algebra $R$ by
algebra automorphisms, and assume that $|G|$ is invertible in
$\kfield$. Define
$$
 P_G:R\longrightarrow R,
 \quad
 P_G(r)=\frac{1}{|G|}\sum_{g\in G}g(r).
$$
Then $P_G$ is an idempotent Reynolds operator of weight $-1$.
\end{example}

The following lemma combines the classical characterization of idempotent
averaging operators~\cite{PeiGuo2015} with a characterization in terms of
Reynolds operators of weight $-1$.

\begin{lemma}
\label{lem:classical-idempotent}
Let $P:R\to R$ be an idempotent linear map, and set
$$
 A:=\im P,\quad K:=\ker P,\quad e:=P(1_R).
$$
Then $R=A\oplus K$, and the following statements are equivalent:
\begin{enumerate}[label=\textup{(\roman*)}]
\item $P$ is a Reynolds operator of weight $-1$.
\item $P$ is an averaging operator, i.e.,
$$
 P(xP(y))=P(x)P(y)=P(P(x)y)
$$
for all $x,y\in R$.
\item $A^2\subseteq A$ and $AK+KA\subseteq K$.
\end{enumerate}
When these conditions hold, $A$ is a unital subalgebra with identity $e$,
and
$$
 A(1_R-e)=(1_R-e)A=0.
$$
\end{lemma}

\begin{proof}
The decomposition $R=A\oplus K$ follows from $P^2=P$. The equivalence
of \textup{(ii)} and \textup{(iii)} follows from
\cite[Proposition~2.3]{PeiGuo2015}.

Suppose that \textup{(ii)} holds. By the averaging identity and $P^2=P$,
we have
$$
 P(P(x)P(y))=P^2(xP(y))=P(xP(y))=P(x)P(y).
$$
Hence
$$
 P(xP(y))+P(P(x)y)-P(P(x)P(y))=P(x)P(y),
$$
so $P$ is a Reynolds operator of weight $-1$. Thus \textup{(i)} holds.

Conversely, suppose that \textup{(i)} holds. For $a,b\in A$ and $k\in K$,
we have $P(a)=a$, $P(b)=b$, and $P(k)=0$. Hence
\begin{align*}
 ab
 =P(a)P(b)=P(aP(b))+P(P(a)b)-P(P(a)P(b))
 =P(ab)
\end{align*}
and
\begin{align*}
 0
 &=P(a)P(k)
 =P(aP(k))+P(P(a)k)-P(P(a)P(k))
 =P(ak),\\
 0
 &=P(k)P(a)
 =P(kP(a))+P(P(k)a)-P(P(k)P(a))
 =P(ka).
\end{align*}
Thus $ab\in A$ and $ak,ka\in K$. Therefore $A^2\subseteq A$, $AK+KA\subseteq K$,
which proves \textup{(iii)}.

Finally, write $1_R=e+k_0$ with $k_0\in K$. For every $a\in A$, $a=a1_R=ae+ak_0$.
Since $ae\in A$ and $ak_0\in K$, the uniqueness of the decomposition
$R=A\oplus K$ gives $ae=a$, $ak_0=0$. Similarly, from
$a=1_Ra=ea+k_0a$ we obtain $ea=a$, $k_0a=0$.
Thus $e$ is the identity of $A$ and $A(1_R-e)=(1_R-e)A=0$.
\end{proof}

We now extend the preceding characterization to idempotent Reynolds operators of arbitrary weight.

\begin{proposition}\label{prop:idempotent-structure}
Let $P:R\to R$ be an idempotent linear map, let
$A=\im P$ and $K=\ker P$, and let $\lambda\in\kfield$. The following
conditions are equivalent:
\begin{enumerate}[label=\textup{(\roman*)}]
\item $P$ is a Reynolds operator of weight $\lambda$.
\item $P$ is an averaging operator and $(1+\lambda)A=0$.
\item $A^2\subseteq A$, $AK+KA\subseteq K$, and $(1+\lambda)A=0$.
\end{enumerate}
\end{proposition}

\begin{proof}
Suppose that \textup{(i)} holds. Take $a,b\in A$. Since
$P(a)=a$ and $P(b)=b$, we have
\begin{align*}
 ab
 &=P(a)P(b)=P\bigl(aP(b)\bigr)+P\bigl(P(a)b\bigr)
   +\lambda P\bigl(P(a)P(b)\bigr)=(2+\lambda)P(ab).
\end{align*}
The right-hand side belongs to $A$, and hence $ab\in A$. Thus
$A^2\subseteq A$, and consequently $P(ab)=ab$. Substituting this into
the above equality gives
\begin{equation}\label{eq:image-square-annihilation}
 (1+\lambda)ab=0.
\end{equation}

Next let $a\in A$ and $k\in K$. Since $P(a)=a$ and $P(k)=0$, we have
\begin{align*}
 0=P(a)P(k)
 &=P\bigl(aP(k)\bigr)+P\bigl(P(a)k\bigr)
   +\lambda P\bigl(P(a)P(k)\bigr)
 =P(ak),\\
 0=P(k)P(a)
 &=P\bigl(kP(a)\bigr)+P\bigl(P(k)a\bigr)
   +\lambda P\bigl(P(k)P(a)\bigr)
 =P(ka).
\end{align*}
Thus $ak,ka\in K$, and therefore $AK+KA\subseteq K$.
By Lemma~\ref{lem:classical-idempotent}, $P$ is an averaging operator
and $e=P(1_R)$ is the identity of $A$. Taking $b=e$ in
Eq.~(\ref{eq:image-square-annihilation}) gives, for every $a\in A$,
$$
 (1+\lambda)a=(1+\lambda)ae=0.
$$
Hence $(1+\lambda)A=0$, and \textup{(ii)} follows.

Conversely, suppose that \textup{(ii)} holds and let $x,y\in R$. Since
$P$ is an averaging operator, $P\bigl(xP(y)\bigr)=P(x)P(y)=P\bigl(P(x)y\bigr)$.
In particular, $P(x)P(y)\in A$. Since $P^2=P$, we have
$$
 P\bigl(P(x)P(y)\bigr)=P\bigl( P(xP(y)) \bigr)=P\bigl(xP(y)\bigr)=P(x)P(y).
$$
Therefore
\begin{align*}
 & P\bigl(xP(y)\bigr)+P\bigl(P(x)y\bigr)
   +\lambda P\bigl(P(x)P(y)\bigr)=(2+\lambda)P(x)P(y)\\
=& P(x)P(y)+(1+\lambda)P(x)P(y)=P(x)P(y),
\end{align*}
where the last equality follows from $(1+\lambda)A=0$.
Thus $P$ satisfies the Reynolds identity of weight $\lambda$, and hence
\textup{(i)} holds.

Finally, the equivalence of \textup{(ii)} and \textup{(iii)} follows
from Lemma~\ref{lem:classical-idempotent}.
\end{proof}

\begin{corollary}\label{cor:idempotent-rigidity}
Suppose that scalar multiplication by $1+\lambda$ is injective on $R$.
Then the zero map is the only idempotent Reynolds operator of weight
$\lambda$ on $R$. In particular, if $\kfield$ is a field, a nonzero
idempotent Reynolds operator can exist only at weight $-1$.
\end{corollary}

\begin{proof}
By Proposition~\ref{prop:idempotent-structure},
$(1+\lambda)\im P=0$. If scalar multiplication by $1+\lambda$ is injective on $R$, then
$\im P=0$, and hence $P=0$. In particular, if $\kfield$ is a field and
$\lambda\neq-1$, then $1+\lambda\neq0$, so scalar multiplication by
$1+\lambda$ is injective on $R$.
\end{proof}

\subsection{Reynolds modules and bimodules}\label{subsec:modules}

\begin{definition}\label{def:left-module}
Let $(R,P)$ be a Reynolds algebra of weight $\lambda$.
\begin{enumerate}[label=\textup{(\alph*)}]
\item A \textbf{left $(R,P)$-module} is a pair $(M,p_M)$
consisting of a left $R$-module $M$ and a linear map
$p_M:M\to M$ such that, for all $r\in R$ and $m\in M$,
\begin{equation}\label{eq:left-module}
 P(r)p_M(m)=p_M(rp_M(m))+p_M(P(r)m)
 +\lambda p_M(P(r)p_M(m)).
\end{equation}

\item Let $(M,p_M)$ and $(N,p_N)$ be left $(R,P)$-modules.
A \textbf{morphism} from $(M,p_M)$ to $(N,p_N)$ is an $R$-module
homomorphism $f:M\to N$ such that $f\circ p_M=p_N\circ f$.
We write $\Hom_{(R,P)}(M,N)$ for the $\kfield$-module of these
morphisms, and $\operatorname{End}_{(R,P)}(M)$ for
$\Hom_{(R,P)}(M,M)$.

\item Let $(M,p_M)$ be a left $(R,P)$-module. A
\textbf{left $(R,P)$-submodule} of $(M,p_M)$ is an $R$-submodule
$N\subseteq M$ such that $p_M(N)\subseteq N$.
\end{enumerate}
We denote the resulting category by ${}_{(R,P)}\mathsf{Mod}$.
\end{definition}

\begin{remark}\label{rem:weight-zero-modules}
When $\lambda=0$, a Reynolds algebra is a Rota-Baxter algebra of weight
zero, and Eq.~(\ref{eq:left-module}) reduces to the left Rota-Baxter
module identity in~\cite{QiaoGaoGuo2019}.
\end{remark}

\begin{example}\label{ex:regular-left}
Let $(R,P)$ be a Reynolds algebra of weight $\lambda$.
\begin{enumerate}[label=\textup{(\alph*)}]
\item With its regular left action, $(R,P)$ is a left $(R,P)$-module
over itself.

\item For every left $R$-module $M$, $(M,0)$ is a left $(R,P)$-module.

\item If $P=\id_R$ and $\lambda=-1$, a linear map
$p_M:M\to M$ makes $(M,p_M)$ a left $(R,P)$-module if and only if
$p_M$ is $R$-linear.
\end{enumerate}
\end{example}

We next consider quotients of left $(R,P)$-modules.

\begin{lemma}\label{lem:quotients}
Let $(M,p_M)$ be a left $(R,P)$-module and let $N\subseteq M$ be a
left $(R,P)$-submodule. Define
$$
 \overline p_M:M/N\longrightarrow M/N,\quad
 \overline p_M(m+N)=p_M(m)+N.
$$
Then $\overline p_M$ is well-defined and makes
$(M/N,\overline p_M)$ a left $(R,P)$-module. We call
$(M/N,\overline p_M)$ the \emph{quotient left $(R,P)$-module} of
$(M,p_M)$ by $N$.
\end{lemma}

\begin{proof}
Since $p_M(N)\subseteq N$, the map $\overline p_M$ is well-defined.
For $r\in R$ and $m\in M$, Eq.~(\ref{eq:left-module}) gives
\begin{align*}
& P(r)\overline p_M(m+N) =P(r)p_M(m)+N\\
=& p_M(rp_M(m))+p_M(P(r)m)
   +\lambda p_M(P(r)p_M(m))+N\\
=& \overline p_M(r\overline p_M(m+N))
   +\overline p_M(P(r)(m+N))
   +\lambda\overline p_M(P(r)\overline p_M(m+N)).
\end{align*}
Thus $(M/N,\overline p_M)$ is a left $(R,P)$-module.
\end{proof}

\begin{example}\label{ex:matrix-module}
Let $(R,P)$ be a Reynolds algebra of weight $\lambda$. For positive
integers $m,n$, let $M_{m,n}(R)$ denote the module of
$m\times n$ matrices over $R$, and define
$$
 P_{m,n}:M_{m,n}(R)\longrightarrow M_{m,n}(R),\quad
 P_{m,n}((a_{ij}))=(P(a_{ij})).
$$
Put $P_m=P_{m,m}$. By
Example~\ref{ex:tensor-reynolds-algebra}\textup{(b)},
$(M_m(R),P_m)$ is a Reynolds algebra of weight $\lambda$. We claim
that $(M_{m,n}(R),P_{m,n})$ is a left $(M_m(R),P_m)$-module under
left matrix multiplication.

Indeed, let $A=(a_{ij})\in M_m(R)$ and
$X=(x_{ij})\in M_{m,n}(R)$. Applying the Reynolds identity in $R$ to
each pair $(a_{ik},x_{kj})$ gives
\begin{align*}
& \bigl[P_m(A)P_{m,n}(X)\bigr]_{ij} =\sum_{k=1}^{m}P(a_{ik})P(x_{kj})\\
=& \sum_{k=1}^{m}\Bigl(
    P\bigl(a_{ik}P(x_{kj})\bigr)
   +P\bigl(P(a_{ik})x_{kj}\bigr)
   +\lambda P\bigl(P(a_{ik})P(x_{kj})\bigr)\Bigr)\\
=& \Bigl[P_{m,n}\bigl(AP_{m,n}(X)\bigr)
    +P_{m,n}\bigl(P_m(A)X\bigr)
    +\lambda P_{m,n}\bigl(P_m(A)P_{m,n}(X)\bigr)\Bigr]_{ij}.
\end{align*}
Thus $(M_{m,n}(R),P_{m,n})$ satisfies Eq.~(\ref{eq:left-module}),
and hence is a left $(M_m(R),P_m)$-module.
\end{example}

\begin{example}\label{ex:group-averaging-module}
Let $(R,P_G)$ be the finite-group Reynolds algebra from
Example~\ref{ex:group-averaging-algebra}. Let $M$ be a left $R$-module
equipped with a compatible linear $G$-action, so that $g(rm)=g(r)g(m)$
for all $g\in G$, $r\in R$, and $m\in M$. Define
$$
 p_M(m)=\frac{1}{|G|}\sum_{g\in G}g(m).
$$
Then $(M,p_M)$ is a left $(R,P_G)$-module.

Indeed, for $h\in G$, reindexing the sum gives $h(p_M(m))=p_M(m)$. Thus $p_M(m)$ is $G$-invariant, while $P_G(r)$ is $G$-invariant by
Example~\ref{ex:group-averaging-algebra}. Hence
\begin{align*}
 p_M(rp_M(m))
 &=\frac{1}{|G|}\sum_{g\in G}g(r)g(p_M(m))
 =\frac{1}{|G|}\sum_{g\in G}g(r)p_M(m)
 =P_G(r)p_M(m),\\
 p_M(P_G(r)m)
 &=\frac{1}{|G|}\sum_{g\in G}g(P_G(r))g(m)
 =P_G(r)p_M(m),\\
 p_M(P_G(r)p_M(m))
 &=P_G(r)p_M(m).
\end{align*}
Consequently,
$$
 P_G(r)p_M(m)
 =p_M(rp_M(m))+p_M(P_G(r)m)-p_M(P_G(r)p_M(m)),
$$
which is Eq.~(\ref{eq:left-module}) for $\lambda=-1$.
\end{example}

For right modules, we distinguish the identity obtained from the opposite
algebra from the identity used in the rest of the paper. Write
$R^{\mathrm{op}}$ for the opposite algebra of $R$. The following
definition introduces both notions.

\begin{definition}\label{def:right-module}
Let $(R,P)$ be a Reynolds algebra of weight $\lambda$, and let $M$ be a
right $R$-module.
\begin{enumerate}[label=\textup{(\alph*)}]
\item A linear map $q_M:M\to M$ satisfies the \textbf{right module
identity induced by the opposite algebra} if
\begin{equation}\label{eq:opposite-right-module}
 q_M(m)P(r)=q_M(mP(r))+q_M(q_M(m)r)
 +\lambda q_M(q_M(m)P(r))
\end{equation}
for all $m\in M$ and $r\in R$.

\item A \textbf{right $(R,P)$-module} is a pair $(M,p_M)$, where
$p_M:M\to M$ is linear and
\begin{equation}\label{eq:right-module}
 p_M(mP(r))=p_M(p_M(m)r)+p_M(m)P(r)
 +\lambda p_M(p_M(m)P(r))
\end{equation}
for all $m\in M$ and $r\in R$.
\end{enumerate}
Morphisms in either class are right $R$-module homomorphisms that commute
with the module operators. We denote by $\mathsf{Mod}_{(R,P)}$ the
category of right $(R,P)$-modules defined in part~\textup{(b)}.
\end{definition}

The convention in part~\textup{(b)} is chosen so that left and right
$(R,P)$-modules are described by left and right actions of the same
ring, as proved in Theorem~\ref{thm:category-equivalence}. At weight
zero, it agrees with the right Rota-Baxter module convention in
\cite{QiaoGaoGuo2019}. The two identities are related as follows.

\begin{proposition}\label{prop:right-sign}
Let $(R,P)$ be a Reynolds algebra of weight $\lambda$.
\begin{enumerate}[label=\textup{(\alph*)}]
\item The pair $(R^{\mathrm{op}},P)$ is a Reynolds algebra of weight
$\lambda$.

\item A pair $(M,q_M)$ satisfies the right module identity induced by the opposite
algebra if and only if $(M,q_M)$ is a left
$(R^{\mathrm{op}},P)$-module.

\item The assignment
$$
 (M,q_M)\longmapsto(M,-q_M)
$$
defines an isomorphism of categories from the category defined by the
right module identity induced by the opposite algebra to $\mathsf{Mod}_{(R,P)}$.

\item In particular, $(R,P)$ satisfies the right module identity induced by the opposite
algebra with respect to the regular right action, while $(R,-P)$ is a
right $(R,P)$-module.
\end{enumerate}
\end{proposition}

\begin{proof}
The Reynolds identity in $R^{\mathrm{op}}$, evaluated at $x$ and $y$, is
the Reynolds identity in $R$ evaluated at $y$ and $x$, with the two linear
terms interchanged. Hence the two identities are equivalent, proving
\textup{(a)}.

Regard a right $R$-module $M$ as a left $R^{\mathrm{op}}$-module by
letting $r\in R^{\mathrm{op}}$ act by $r\cdot m=mr$. The left module
identity over $(R^{\mathrm{op}},P)$ becomes
$$
 q_M(m)P(r)
 =q_M(q_M(m)r)+q_M(mP(r))
 +\lambda q_M(q_M(m)P(r)),
$$
which is Eq.~(\ref{eq:opposite-right-module}). This proves
\textup{(b)}.

Now put $p_M=-q_M$. Substituting into
Eq.~(\ref{eq:opposite-right-module}) gives
$$
 -p_M(m)P(r)
 =-p_M(mP(r))+p_M(p_M(m)r)
 +\lambda p_M(p_M(m)P(r)).
$$
Rearranging yields Eq.~(\ref{eq:right-module}). The same computation with
$q_M=-p_M$ gives the converse. Since the sign change does not alter
morphisms and is its own inverse, it defines an isomorphism of categories.
This proves \textup{(c)}.

By Example~\ref{ex:regular-left}, $(R^{\mathrm{op}},P)$ is a left module
over itself. Parts~\textup{(b)} and~\textup{(c)} therefore imply
\textup{(d)}.
\end{proof}

Proposition~\ref{prop:right-sign} shows that $(R,-P)$ is always a right
$(R,P)$-module. We now determine when $(R,P)$ itself is a right
$(R,P)$-module.

\begin{proposition}\label{prop:regular-right}
Let $(R,P)$ be a Reynolds algebra of weight $\lambda$. With the regular
right $R$-action and module operator $P$, the pair $(R,P)$ is a right
$(R,P)$-module if and only if
\begin{equation}\label{eq:regular-right-condition}
 2P(P(r)s)+2\lambda P(P(r)P(s))=0
\end{equation}
for all $r,s\in R$. If $2$ is invertible in $\kfield$, this is
equivalent to $P(P(r)s)+\lambda P(P(r)P(s))=0$.
\end{proposition}

\begin{proof}
By Definition~\ref{def:right-module}, $(R,P)$ is a right
$(R,P)$-module if and only if
\begin{equation}\label{eq:regular-right-identity}
 P(rP(s))
 =P(P(r)s)+P(r)P(s)+\lambda P(P(r)P(s))
\end{equation}
for all $r,s\in R$. On the other hand, the Reynolds identity in $R$
gives
\begin{equation}\label{eq:regular-right-substitution}
 P(r)P(s)
 =P(rP(s))+P(P(r)s)+\lambda P(P(r)P(s)).
\end{equation}
Substituting Eq.~(\ref{eq:regular-right-substitution}) into
Eq.~(\ref{eq:regular-right-identity}) gives
$$
 P(rP(s))
 =P(rP(s))+2P(P(r)s)+2\lambda P(P(r)P(s)).
$$
Cancelling $P(rP(s))$ from both sides gives $2P(P(r)s)+2\lambda P(P(r)P(s))=0$.

Conversely, suppose that Eq.~(\ref{eq:regular-right-condition}) holds.
Using Eq.~(\ref{eq:regular-right-substitution}), we obtain
\begin{align*}
 P(P(r)s)+P(r)P(s)+\lambda P(P(r)P(s))
 &=P(rP(s))+2P(P(r)s)+2\lambda P(P(r)P(s))=P(rP(s)),
\end{align*}
where the last equality follows from
Eq.~(\ref{eq:regular-right-condition}). Thus
Eq.~(\ref{eq:regular-right-identity}) holds, so $(R,P)$ is a right $(R,P)$-module.

If $2$ is invertible in $\kfield$, the last assertion
follows immediately.
\end{proof}

\begin{example} \label{ex:regular-right-failure}
Let $R=\mathbb{Q}[x]$. Consider the $\mathbb{Q}$-linear operator $P:R\to R$ defined by
$$
 P(x^n)=\frac{1}{n+1}x^n,\quad n\geq0.
$$
It is known that $P$ is a Reynolds operator of weight $-1$ on $R$
\cite{ZhangGaoGuo2021}. However, $(R,P)$ is not a right
$(R,P)$-module with respect to the regular right action.

Since $2$ is invertible in $\mathbb{Q}$, by
Proposition~\ref{prop:regular-right}, it suffices to find $r,s\in R$
such that
$$
 P(P(r)s)-P(P(r)P(s))\neq0.
$$
Taking $r=s=x$, we have $P(x)=x/2$ and $P(x^2)=x^2/3$. Thus
$$
 P(P(x)x)=P\left(\frac{x^2}{2}\right)=\frac{x^2}{6},
 \quad
 P(P(x)P(x))=P\left(\frac{x^2}{4}\right)=\frac{x^2}{12}.
$$
Therefore
$$
 P(P(x)x)-P(P(x)P(x))=\frac{x^2}{12}\neq0,
$$
as required.
\end{example}

Motivated by the notion of a Rota-Baxter bimodule
in~\cite{QiaoGaoGuo2019}, we introduce the following
notion of a bimodule for Reynolds algebras.

\begin{definition}\label{def:bimodule}
Let $(S,T)$ and $(R,P)$ be Reynolds algebras of weights $\mu$ and
$\lambda$, respectively. An \textbf{$(S,T)$-$(R,P)$-bimodule} is a
triple $(B,p_B^S,p_B^R)$ satisfying the following conditions:
\begin{enumerate}[label=\textup{(\roman*)}]
\item $B$ is an $S$-$R$-bimodule.
\item $(B,p_B^S)$ is a left $(S,T)$-module.
\item $(B,p_B^R)$ is a right $(R,P)$-module.
\item for all $s\in S$, $r\in R$, and $b\in B$,
$$
 p_B^R(sb)=s p_B^R(b),\quad
 p_B^S(br)=p_B^S(b)r,\quad
 p_B^S(p_B^R(b))=p_B^R(p_B^S(b)).
$$
\end{enumerate}
A \textbf{morphism of $(S,T)$-$(R,P)$-bimodules} is an
$S$-$R$-bimodule homomorphism that commutes with both module operators.
\end{definition}

\subsection{The Reynolds double product and module structures}
\label{subsec:double}

We first recall the Reynolds double product introduced in
\cite{ZhangGaoGuo2021} and then extend this construction to left $(R,P)$-modules.

\begin{proposition}[{\cite{ZhangGaoGuo2021}}]\label{prop:double-product}
Let $(R,P)$ be a Reynolds algebra of weight $\lambda$. Define
$$
 r\star_P s:=rP(s)+P(r)s+\lambda P(r)P(s),
 \quad r,s\in R.
$$
Then $\star_P$ is associative, although it need not be unital, and
\begin{equation}\label{eq:P-multiplicative}
 P(r\star_P s)=P(r)P(s).
\end{equation}
Moreover, $P$ is a Reynolds operator of weight $\lambda$ on the
possibly nonunital algebra $R^{(P)}:=(R,\star_P)$.
\end{proposition}

For the possibly nonunital algebra $R^{(P)}$, a left module is not
required to be unital. With this convention, the double product extends
naturally to left $(R,P)$-modules.

\begin{proposition}\label{prop:double-module}
Let $(M,p_M)$ be a left $(R,P)$-module. Define
$$
 r\star_{p_M}m
 :=rp_M(m)+P(r)m+\lambda P(r)p_M(m),
 \quad r\in R,\ m\in M.
$$
Then $M^{(p_M)}:=(M,\star_{p_M})$ is a left $R^{(P)}$-module,
$p_M(r\star_{p_M}m)=P(r)p_M(m)$ for all $r\in R$ and $m\in M$,
and $(M^{(p_M)},p_M)$ is a left $(R^{(P)},P)$-module.
\end{proposition}

\begin{proof}
The identity $p_M(r\star_{p_M}m)=P(r)p_M(m)$ follows immediately
from Eq.~(\ref{eq:left-module}) and the definition of $\star_{p_M}$.
We first verify the module associativity. By
Eq.~(\ref{eq:P-multiplicative}),
\begin{align*}
& (r\star_P s)\star_{p_M}m
=(r\star_P s)p_M(m)+P(r\star_P s)m
   +\lambda P(r\star_P s)p_M(m)\\
=& \bigl(rP(s)+P(r)s+\lambda P(r)P(s)\bigr)p_M(m)
   +P(r)P(s)m+\lambda P(r)P(s)p_M(m)\\
=& rP(s)p_M(m)+P(r)sp_M(m)+P(r)P(s)m
   +2\lambda P(r)P(s)p_M(m).
\end{align*}
On the other hand, using this identity, we have
\begin{align*}
& r\star_{p_M}(s\star_{p_M}m)
=r p_M(s\star_{p_M}m)+P(r)(s\star_{p_M}m)
   +\lambda P(r)p_M(s\star_{p_M}m)\\
=& rP(s)p_M(m)
   +P(r)\bigl(sp_M(m)+P(s)m+\lambda P(s)p_M(m)\bigr)
   +\lambda P(r)P(s)p_M(m)\\
=& rP(s)p_M(m)+P(r)sp_M(m)+P(r)P(s)m
   +2\lambda P(r)P(s)p_M(m).
\end{align*}
Thus $(r\star_P s)\star_{p_M}m
 =r\star_{p_M}(s\star_{p_M}m)$,
so $M^{(p_M)}$ is a left $R^{(P)}$-module.

It remains to verify Eq.~(\ref{eq:left-module}) for
$(M^{(p_M)},p_M)$ over $(R^{(P)},P)$. The left-hand side is
$$
 P(r)\star_{p_M}p_M(m)
 =P(r)p_M^2(m)+P^2(r)p_M(m)
   +\lambda P^2(r)p_M^2(m).
$$
Using the same identity, the right-hand side is
\begin{align*}
 &p_M(r\star_{p_M}p_M(m))
 +p_M(P(r)\star_{p_M}m)
 +\lambda p_M(P(r)\star_{p_M}p_M(m))\\
=& P(r)p_M^2(m)+P^2(r)p_M(m)
   +\lambda P^2(r)p_M^2(m).
\end{align*}
Hence $(M^{(p_M)},p_M)$ is a left $(R^{(P)},P)$-module.
\end{proof}

\subsection{Bijective Reynolds modules and differential modules}
\label{subsec:invertible}

We next consider Reynolds modules with bijective module operators and
relate them to differential modules. For a bijective Reynolds operator,
the corresponding relation with derivations is classical in weight
$-1$~\cite{ZhangGaoGuo2021}.
We use the standard notions of differential algebras and differential
modules~\cite{QiaoPei2026}.

Recall that a \textbf{differential algebra} is a pair $(R,\delta)$,
where $\delta:R\to R$ is a derivation. A
\textbf{left differential module} over $(R,\delta)$ is a pair
$(M,d)$ consisting of a left $R$-module $M$ and a linear map
$d:M\to M$ satisfying
$d(rm)=\delta(r)m+rd(m)$ for all $r\in R$ and $m\in M$.
A morphism of left differential modules is an $R$-module homomorphism
that commutes with the module operators.

Let $(R,P)$ be a Reynolds algebra of weight $\lambda$ and suppose that
$P$ is bijective. Set $\delta:=P^{-1}+\lambda\id_R$.
Substituting $x=P^{-1}(a)$ and $y=P^{-1}(b)$ into
Eq.~(\ref{eq:weighted-reynolds}) and applying $P^{-1}$ gives
$$
 P^{-1}(ab)
 =P^{-1}(a)b+aP^{-1}(b)+\lambda ab.
$$
Hence
$$
 \delta(ab)
 =\delta(a)b+a\delta(b),
$$
so $\delta$ is a derivation of $R$.

The corresponding statement for modules is as follows.

\begin{theorem}\label{thm:invertible-equivalence}
Let $(R,P)$ be a Reynolds algebra of weight $\lambda$ with $P$
bijective, and let $\delta:=P^{-1}+\lambda\id_R$.
Then the assignment
$$
 (M,p_M)\longmapsto
 (M,p_M^{-1}+\lambda\id_M)
$$
defines an equivalence between the full subcategory of left
$(R,P)$-modules with bijective module operator and the full
subcategory of left differential $(R,\delta)$-modules $(M,d)$
for which $d-\lambda\id_M$ is bijective.
\end{theorem}

\begin{proof}
Let $(M,p_M)$ be a left $(R,P)$-module with $p_M$ bijective, and set
$d:=p_M^{-1}+\lambda\id_M$.
We first show that $(M,d)$ is a left differential $(R,\delta)$-module.

For $a\in R$ and $n\in M$, take $r=P^{-1}(a)$ and
$m=p_M^{-1}(n)$ in Eq.~(\ref{eq:left-module}). Since $P(r)=a$ and
$p_M(m)=n$, we obtain
$$
 an
 =p_M(P^{-1}(a)n)+p_M(ap_M^{-1}(n))
  +\lambda p_M(an).
$$
Applying $p_M^{-1}$ gives
$$
 p_M^{-1}(an)
 =P^{-1}(a)n+a p_M^{-1}(n)+\lambda an.
$$
Therefore
\begin{align*}
d(an) =p_M^{-1}(an)+\lambda an=P^{-1}(a)n+a p_M^{-1}(n)+2\lambda an =\delta(a)n+a d(n).
\end{align*}
Thus $(M,d)$ is a left differential $(R,\delta)$-module. Moreover,
$d-\lambda\id_M=p_M^{-1}$ is bijective.

Conversely, let $(M,d)$ be a left differential $(R,\delta)$-module
such that $d-\lambda\id_M$ is bijective, and set
$p_M:=(d-\lambda\id_M)^{-1}$.
Then $p_M^{-1}=d-\lambda\id_M$. For $a\in R$ and $n\in M$, the
differential module identity gives
\begin{align*}
 & p_M^{-1}(an) =d(an)-\lambda an=\delta(a)n+a d(n)-\lambda an=P^{-1}(a)n+a p_M^{-1}(n)+\lambda an.
\end{align*}
Taking $a=P(r)$ and $n=p_M(m)$, we get
$$
 p_M^{-1}\bigl(P(r)p_M(m)\bigr)
 =rp_M(m)+P(r)m+\lambda P(r)p_M(m).
$$
Applying $p_M$ yields
$$
 P(r)p_M(m)
 =p_M(rp_M(m))+p_M(P(r)m)
  +\lambda p_M(P(r)p_M(m)),
$$
which is Eq.~(\ref{eq:left-module}). Hence $(M,p_M)$ is a left
$(R,P)$-module with bijective module operator.

It remains to compare morphisms. Let $(M,p_M)$ and $(N,p_N)$ be left
$(R,P)$-modules with bijective module operators, and let
$f:M\to N$ be an $R$-module homomorphism. Then
$f\circ p_M=p_N\circ f$ if and only if $f\circ p_M^{-1}=p_N^{-1}\circ f$,
and this is equivalent to
$$
 f\circ(p_M^{-1}+\lambda\id_M)
 =(p_N^{-1}+\lambda\id_N)\circ f.
$$
Thus $f$ is a morphism of left $(R,P)$-modules if and only if it is
a morphism of the corresponding differential modules.

The defining formulas show that the two constructions are mutually
inverse on objects. Together with the preceding characterization of
morphisms, they define the stated equivalence.
\end{proof}

\section{Free Reynolds modules}\label{sec:free}

In this section, we recall free operated modules and construct free left $(R,P)$-modules. We also discuss a restricted universal property for ordinary free $R$-modules.

\subsection{Free operated modules and free Reynolds modules}
\label{subsec:free}

We first recall operated modules and their free objects
from~\cite{QiaoGaoGuo2019}. Throughout this subsection, let $(R,P)$ be
a Reynolds algebra of weight $\lambda$. The construction of free operated
modules only uses the algebra $R$, while the Reynolds operator $P$ enters
through the quotient relations below.

\begin{definition}\label{def:operated-module}
A \textbf{left operated $R$-module} is a pair $(M,q)$ consisting of a
left $R$-module $M$ and a linear map $q:M\to M$. A morphism of left
operated $R$-modules is an $R$-module homomorphism that commutes with
the operators. An \textbf{operated $R$-submodule} of $(M,q)$ is an $R$-submodule
$N\subseteq M$ such that $q(N)\subseteq N$.
\end{definition}

\begin{definition}\label{def:free-operated-module}
Let $X$ be a set. A \textbf{free left operated $R$-module on $X$} is a
left operated $R$-module $(F,q_F)$ together with a set map
$j:X\to F$ such that, for every left operated $R$-module $(M,q)$ and
every set map $f:X\to M$, there exists a unique morphism of left operated
$R$-modules $\widetilde f:(F,q_F)\rightarrow(M,q)$
such that $\widetilde f\circ j=f$.
\end{definition}

We now recall the explicit construction of the free left operated
$R$-module in~\cite{QiaoGaoGuo2019}. Let $X$ be a set and let $\kfield X$ be the free
module on $X$. Define
$$
 \mathcal{M}_R(X):=
 \bigoplus_{n\geq1}R^{\otimes n}\otimes\kfield X.
$$
The algebra $R$ acts on the first tensor factor. For $r_1,\ldots,r_n\in R$ and $x\in X$, write
$(r_1\otimes\cdots\otimes r_n)x$ for
$(r_1\otimes\cdots\otimes r_n)\otimes x$.
Define $q_X:\mathcal{M}_R(X)\to\mathcal{M}_R(X)$ by
$$
 q_X\bigl((r_1\otimes\cdots\otimes r_n)x\bigr)
 =(1_R\otimes r_1\otimes\cdots\otimes r_n)x,
$$
and define $j_X:X\to\mathcal{M}_R(X)$ by $j_X(x)=1_Rx$.

\begin{proposition}[{\cite{QiaoGaoGuo2019}}]\label{prop:free-operated}
The triple $(\mathcal{M}_R(X),q_X,j_X)$ is the free left operated
$R$-module on $X$.
\end{proposition}

We introduce the notion of a free left $(R,P)$-module as follows.

\begin{definition}\label{def:free-reynolds-module}
Let $X$ be a set. A \textbf{free left $(R,P)$-module on $X$} is a
left $(R,P)$-module $(F,p_F)$ together with a set map $\iota:X\to F$
such that, for every left $(R,P)$-module $(M,p_M)$ and every set map
$f:X\to M$, there exists a unique morphism of left $(R,P)$-modules
$\widetilde f:F\to M$ such that $\widetilde f\circ\iota=f$.
\end{definition}

We now construct the free left $(R,P)$-module on $X$ from the free
operated $R$-module $\mathcal M_R(X)$. Let $I_X$ be the smallest
operated $R$-submodule of $\mathcal M_R(X)$ containing all elements
$$
 P(r)q_X(u)-q_X(rq_X(u))-q_X(P(r)u)
 -\lambda q_X(P(r)q_X(u)),
 \quad r\in R,\ u\in\mathcal M_R(X).
$$
Since $I_X$ is stable under $q_X$, the operator $q_X$ induces an
operator $\overline q_X$ on
$F_{R,P}(X):=\mathcal M_R(X)/I_X$.
Let $\pi_X:\mathcal M_R(X)\to F_{R,P}(X)$ be the quotient map and set
$\iota_X:=\pi_X\circ j_X$.

\begin{theorem}\label{thm:free-reynolds}
The pair $(F_{R,P}(X),\overline q_X)$, together with the map
$\iota_X:X\to F_{R,P}(X)$, is the free left $(R,P)$-module on $X$.
\end{theorem}

\begin{proof}
Since $I_X$ is stable under $q_X$, the induced operator is given by
$\overline q_X(\pi_X(u))=\pi_X(q_X(u))$. By the definition of $I_X$,
for all $r\in R$ and $u\in\mathcal M_R(X)$, we have
\begin{align*}
 P(r)\overline q_X(\pi_X(u))
 =\overline q_X(r\overline q_X(\pi_X(u)))
   +\overline q_X(P(r)\pi_X(u))+\lambda\overline q_X(P(r)\overline q_X(\pi_X(u))).
\end{align*}
Since every element of $F_{R,P}(X)$ is of the form $\pi_X(u)$, this is
Eq.~(\ref{eq:left-module}). Thus
$(F_{R,P}(X),\overline q_X)$ is a left $(R,P)$-module.

Let $(M,p_M)$ be a left $(R,P)$-module and let $f:X\to M$ be a set map.
By Proposition~\ref{prop:free-operated}, there exists a unique morphism
of left operated $R$-modules
$\widetilde f:\mathcal M_R(X)\to M$ such that
$\widetilde f\circ j_X=f$. For all $r\in R$ and
$u\in\mathcal M_R(X)$,
\begin{align*}
 & \widetilde f\bigl(
 P(r)q_X(u)-q_X(rq_X(u))-q_X(P(r)u)
 -\lambda q_X(P(r)q_X(u))\bigr)\\*
=& P(r)p_M(\widetilde f(u))
   -p_M(rp_M(\widetilde f(u)))
   -p_M(P(r)\widetilde f(u))-\lambda p_M(P(r)p_M(\widetilde f(u)))=0
\end{align*}
by Eq.~(\ref{eq:left-module}). Since $\widetilde f$ is a morphism of
operated $R$-modules, its kernel is an operated $R$-submodule.
Therefore $I_X\subseteq\ker\widetilde f$, and $\widetilde f$ factors
uniquely through an $R$-linear map
$\overline f:F_{R,P}(X)\to M$ such that
$\overline f\circ\pi_X=\widetilde f$.

Since $\widetilde f$ commutes with the operators, so does $\overline f$.
Thus $\overline f$ is a morphism of left $(R,P)$-modules. Moreover,
$$
 \overline f\circ\iota_X
 =\overline f\circ\pi_X\circ j_X
 =\widetilde f\circ j_X
 =f.
$$

For uniqueness, let $g:F_{R,P}(X)\to M$ be a morphism of left
$(R,P)$-modules such that $g\circ\iota_X=f$. Then
$g\circ\pi_X:\mathcal M_R(X)\to M$ is a morphism of left operated
$R$-modules and
$$
 (g\circ\pi_X)\circ j_X
 =g\circ\iota_X=f.
$$
By Proposition~\ref{prop:free-operated},
$g\circ\pi_X=\widetilde f=\overline f\circ\pi_X$. Since $\pi_X$ is
surjective, $g=\overline f$.
\end{proof}

\begin{corollary}\label{cor:free-quotient}
Every left $(R,P)$-module is a quotient of a free left $(R,P)$-module.
\end{corollary}

\begin{proof}
Let $(M,p_M)$ be a left $(R,P)$-module. Applying
Theorem~\ref{thm:free-reynolds} to the underlying set of $M$ and the
identity map $\id_M:M\to M$ gives a morphism of left $(R,P)$-modules
$\varepsilon:F_{R,P}(M)\to M$ such that
$\varepsilon\circ\iota_M=\id_M$. Hence $\varepsilon$ is surjective.
\end{proof}

\subsection{Ordinary free modules with a restricted universal property}
\label{subsec:restricted-free}

Following the corresponding construction for Rota-Baxter modules in
\cite{QiaoGaoGuo2019}, we consider a restricted universal property of
the ordinary free $R$-module.

Let $(M,p_M)$ be a left $(R,P)$-module. The set
$$
 M_C(M):=\{m\in M\mid p_M(rm)=P(r)m\text{ for all }r\in R\}
$$
is called the \textbf{set of module constants} of $M$.

For a set $X$, let $R^{(X)}$ be the free left $R$-module with basis
$X$. Define
$$
 P_X\Bigl(\sum_{x\in X}r_xx\Bigr)
 :=\sum_{x\in X}P(r_x)x.
$$

\begin{theorem}\label{thm:restricted-free}
The pair $(R^{(X)},P_X)$ is a left $(R,P)$-module. Moreover, for every
left $(R,P)$-module $(M,p_M)$ and every set map $f:X\to M_C(M)$, there
exists a unique morphism of left $(R,P)$-modules
$\overline f:(R^{(X)},P_X)\to(M,p_M)$ extending $f$.
\end{theorem}

\begin{proof}
The left $(R,P)$-module identity for $(R^{(X)},P_X)$ follows
coefficientwise from Eq.~(\ref{eq:weighted-reynolds}). Indeed, for
$r,s\in R$ and $x\in X$,
\begin{align*}
 & P(r)P_X(sx)=P(r)P(s)x\\
=& P(rP(s))x+P(P(r)s)x+\lambda P(P(r)P(s))x\\
=& P_X(rP_X(sx))+P_X(P(r)sx)
   +\lambda P_X(P(r)P_X(sx)).
\end{align*}
By linearity, the identity holds for every element of $R^{(X)}$.

By the universal property of the free $R$-module $R^{(X)}$, there is a
unique $R$-linear map $\overline f:R^{(X)}\to M$ extending $f$, given
by $\overline f(\sum r_xx)=\sum r_xf(x)$. For $x\in X$ and $r\in R$,
since $f(x)\in M_C(M)$, we have
$$
 p_M(\overline f(rx))
 =p_M(rf(x))
 =P(r)f(x)
 =\overline f(P(r)x)
 =\overline f(P_X(rx)).
$$
By linearity, $p_M\circ\overline f=\overline f\circ P_X$. Hence
$\overline f$ is a morphism of left $(R,P)$-modules. Its uniqueness
follows from the uniqueness of the underlying $R$-linear map.
\end{proof}

\begin{corollary}\label{cor:regular-universal}
Let $(M,p_M)$ be a left $(R,P)$-module. For every $m\in M_C(M)$,
there exists a unique morphism of left $(R,P)$-modules
$f:(R,P)\to(M,p_M)$ such that $f(1_R)=m$.
\end{corollary}

\begin{proof}
Apply Theorem~\ref{thm:restricted-free} with $X$ a singleton.
\end{proof}

\section{The ring of Reynolds operators}\label{sec:enveloping}

In this section, we construct the ring of Reynolds operators and identify its left and right module categories. We then give an explicit graded description of this ring for idempotent Reynolds operators of weight $-1$ that fix the identity.

\subsection{Construction and module categories}\label{subsec:enveloping}

Let $(R,P)$ be a Reynolds algebra of weight $\lambda$. Let
$\kfield[Q]$ be the polynomial algebra in one indeterminate $Q$.
Consider the free product
$$
 \mathcal F_R:=R*_{\kfield}\kfield[Q].
$$
Let $\iota_R:R\to\mathcal F_R$ be the canonical algebra homomorphism.
We use the same symbol $Q$ for the image of the generator in
$\mathcal F_R$.

\begin{definition}\label{def:enveloping}
The \textbf{ring of Reynolds operators} on $(R,P)$, denoted by
$\U_{R,P}$, is the quotient
$$
 \U_{R,P}:=\mathcal F_R/I_{R,P},
$$
where $I_{R,P}$ is the two-sided ideal generated by
\begin{equation}\label{eq:enveloping-relation}
 \iota_R(P(r))Q-Q\iota_R(P(r))
 -Q\iota_R(r+\lambda P(r))Q,
 \quad r\in R.
\end{equation}
\end{definition}

We continue to denote the image of the generator in $\U_{R,P}$ by $Q$.
The ring of Reynolds operators has the following universal property.

\begin{proposition}
\label{prop:enveloping-universal}
Let $\pi:\mathcal F_R\to\U_{R,P}$ be the quotient map and set
$j_R:=\pi\circ\iota_R$. Then:
\begin{enumerate}[label=\textup{(\alph*)}]
\item The algebra homomorphism $j_R:R\to\U_{R,P}$ is injective.

\item Let $A$ be an algebra, let
$\alpha:R\to A$ be a unital algebra homomorphism, and let $q\in A$.
If
\begin{equation}\label{eq:enveloping-universal-relation}
 \alpha(P(r))q-q\alpha(P(r))
 -q\alpha(r+\lambda P(r))q=0
\end{equation}
for all $r\in R$, then there exists a unique unital algebra
homomorphism $\widetilde\alpha:\U_{R,P}\to A$ such that
$\widetilde\alpha\circ j_R=\alpha$ and
$\widetilde\alpha(Q)=q$.
\end{enumerate}
\end{proposition}

\begin{proof}
For \textup{(a)}, let
$L_r\in\operatorname{End}_{\kfield}(R)$ be defined by
$L_r(x)=rx$.
The map $r\mapsto L_r$ is a unital algebra homomorphism.
By the universal property of the free product, there exists a unique
unital algebra homomorphism
$\rho:\mathcal F_R\to\operatorname{End}_{\kfield}(R)$ such that
$\rho(\iota_R(r))=L_r$ for all $r\in R$ and $\rho(Q)=P$.

For all $r,x\in R$, we have
\begin{align*}
 &\rho\bigl(
 \iota_R(P(r))Q-Q\iota_R(P(r))
 -Q\iota_R(r+\lambda P(r))Q
 \bigr)(x)\\
 =&
 P(r)P(x)-P(P(r)x)-P(rP(x))
 -\lambda P(P(r)P(x))=0,
\end{align*}
where the last equality follows from
Eq.~(\ref{eq:weighted-reynolds}).
Thus $\rho$ sends every generator of $I_{R,P}$ to zero, and hence
$I_{R,P}\subseteq\ker\rho$.
Therefore $\rho$ induces a unital algebra homomorphism
$\overline\rho:\U_{R,P}\to\operatorname{End}_{\kfield}(R)$
satisfying $\overline\rho\circ\pi=\rho$.

If $j_R(r)=0$, then
$L_r=\rho(\iota_R(r))=\overline\rho(j_R(r))=0$.
Evaluating at $1_R$ gives $r=L_r(1_R)=0$.
Therefore $j_R$ is injective.

For \textup{(b)}, by the universal property of the free product,
there exists a unique unital algebra homomorphism
$\widehat\alpha:\mathcal F_R\to A$ such that
$\widehat\alpha\circ\iota_R=\alpha$ and
$\widehat\alpha(Q)=q$.
By Eq.~(\ref{eq:enveloping-universal-relation}),
$\widehat\alpha$ sends every generator of $I_{R,P}$ to zero.
Hence $I_{R,P}\subseteq\ker\widehat\alpha$, so
$\widehat\alpha$ induces a unital algebra homomorphism
$\widetilde\alpha:\U_{R,P}\to A$ satisfying
$\widetilde\alpha\circ\pi=\widehat\alpha$.

Consequently,
$\widetilde\alpha\circ j_R=\alpha$ and
$\widetilde\alpha(Q)=q$.
Since $\U_{R,P}$ is generated by $j_R(R)$ and $Q$, these conditions
uniquely determine $\widetilde\alpha$.
\end{proof}

By Proposition~\ref{prop:enveloping-universal}\textup{(a)}, we
henceforth identify $R$ with its image in $\U_{R,P}$.
By the definition of $I_{R,P}$, the following identity holds in
$\U_{R,P}$:
\begin{equation}\label{eq:enveloping-compact}
 P(r)Q=QP(r)+Q(r+\lambda P(r))Q,
 \quad r\in R.
\end{equation}
We next relate the rings of Reynolds operators on $(R,P)$ and
$(R^{\mathrm{op}},P)$.

\begin{proposition}\label{prop:opposite-enveloping}
Let $Q_{\mathrm{op}}$ denote the operator generator of
$\U_{R^{\mathrm{op}},P}$. There is a canonical algebra isomorphism
$\U_{R^{\mathrm{op}},P}\to\U_{R,P}^{\mathrm{op}}$
determined by $r\mapsto r$ for all $r\in R^{\mathrm{op}}$ and
$Q_{\mathrm{op}}\mapsto -Q$.
\end{proposition}

\begin{proof}
By Proposition~\ref{prop:right-sign}\textup{(a)},
$(R^{\mathrm{op}},P)$ is a Reynolds algebra of weight $\lambda$.
The canonical embedding $R\to\U_{R,P}$ induces a unital algebra
homomorphism
$R^{\mathrm{op}}\to\U_{R,P}^{\mathrm{op}}$ given by $r\mapsto r$.

Let $\cdot_{\mathrm{op}}$ denote the multiplication in
$\U_{R,P}^{\mathrm{op}}$. For every $r\in R$, we have
\begin{align*}
 &P(r)\mathbin{\cdot_{\mathrm{op}}}(-Q)
  -(-Q)\mathbin{\cdot_{\mathrm{op}}}P(r)
  -(-Q)\mathbin{\cdot_{\mathrm{op}}}(r+\lambda P(r))
       \mathbin{\cdot_{\mathrm{op}}}(-Q)\\
 =& -QP(r)+P(r)Q-Q(r+\lambda P(r))Q=0,
\end{align*}
where the unmarked products are taken in $\U_{R,P}$ and the last
equality follows from Eq.~(\ref{eq:enveloping-compact}).
Thus Proposition~\ref{prop:enveloping-universal}\textup{(b)}
gives a unital algebra homomorphism
$\U_{R^{\mathrm{op}},P}\to\U_{R,P}^{\mathrm{op}}$
sending $r$ to $r$ and $Q_{\mathrm{op}}$ to $-Q$.

Applying the same construction to $R^{\mathrm{op}}$ gives a
homomorphism
$\U_{R,P}\to\U_{R^{\mathrm{op}},P}^{\mathrm{op}}$
sending $r$ to $r$ and $Q$ to $-Q_{\mathrm{op}}$.
Taking opposite algebras yields a homomorphism
$\U_{R,P}^{\mathrm{op}}\to\U_{R^{\mathrm{op}},P}$
with the same values on the generators.
The composites in both orders fix every algebra generator.
Hence the two homomorphisms are mutually inverse.
\end{proof}

For an algebra $C$, write $C\text{-}\mathsf{Mod}$ and
$\mathsf{Mod}\text{-}C$ for the categories of left and right
$C$-modules, respectively.

\begin{theorem}\label{thm:category-equivalence}
Let $(R,P)$ be a Reynolds algebra of weight $\lambda$.
\begin{enumerate}[label=\textup{(\alph*)}]
\item The category ${}_{(R,P)}\mathsf{Mod}$ is naturally equivalent to
$\U_{R,P}\text{-}\mathsf{Mod}$.

\item The category $\mathsf{Mod}_{(R,P)}$ is naturally equivalent to
$\mathsf{Mod}\text{-}\U_{R,P}$.

\item Let $(S,T)$ be a Reynolds algebra of weight $\mu$.
Under these equivalences, $(S,T)$-$(R,P)$-bimodules correspond to
$\U_{S,T}$-$\U_{R,P}$-bimodules.
\end{enumerate}
\end{theorem}

\begin{proof}
For \textup{(a)}, let $(M,p_M)$ be a left $(R,P)$-module.
The left $R$-action determines a unital algebra homomorphism
$R\to\operatorname{End}_{\kfield}(M)$.
By Eq.~(\ref{eq:left-module}), this homomorphism and the endomorphism
$p_M$ satisfy Eq.~(\ref{eq:enveloping-universal-relation}).
Thus Proposition~\ref{prop:enveloping-universal}\textup{(b)}
gives a unique left $\U_{R,P}$-action extending the $R$-action and
satisfying $Qm=p_M(m)$ for all $m\in M$.

Conversely, let $M$ be a left $\U_{R,P}$-module.
Restrict its action to $R$ and define $p_M(m)=Qm$.
For all $r\in R$ and $m\in M$, Eq.~(\ref{eq:enveloping-compact})
gives
\begin{align*}
 & P(r)p_M(m)
=\bigl(QP(r)+Q(r+\lambda P(r))Q\bigr)m\\
=& p_M(P(r)m)+p_M(rp_M(m))
   +\lambda p_M(P(r)p_M(m)).
\end{align*}
Hence $(M,p_M)$ is a left $(R,P)$-module.

Since $\U_{R,P}$ is generated by $R$ and $Q$, an $R$-module
homomorphism $f:M\to N$ is $\U_{R,P}$-linear if and only if
$f\circ p_M=p_N\circ f$.
The two constructions are mutually inverse on objects and leave the
underlying maps unchanged. Hence they define mutually inverse
functors, proving \textup{(a)}.

For \textup{(b)}, let $(M,p_M)$ be a right $(R,P)$-module.
By the universal property of the free product, the right $R$-action
extends uniquely to a right $\mathcal F_R$-action such that
$mQ=p_M(m)$.
For all $r\in R$ and $m\in M$, we have
\begin{align*}
 &m\bigl(\iota_R(P(r))Q-Q\iota_R(P(r))
       -Q\iota_R(r+\lambda P(r))Q\bigr)\\
=& p_M(mP(r))-p_M(m)P(r)-p_M(p_M(m)r)
   -\lambda p_M(p_M(m)P(r))=0,
\end{align*}
where the last equality follows from Eq.~(\ref{eq:right-module}).
Thus every generator of $I_{R,P}$ acts trivially, and the action
descends to a right $\U_{R,P}$-action.

Conversely, let $M$ be a right $\U_{R,P}$-module.
Restrict its action to $R$ and define $p_M(m)=mQ$.
Applying Eq.~(\ref{eq:enveloping-compact}) to $m$ on the right gives
Eq.~(\ref{eq:right-module}). Hence $(M,p_M)$ is a right
$(R,P)$-module.
A right $R$-module homomorphism is $\U_{R,P}$-linear if and only if
it commutes with the module operators.
As in \textup{(a)}, these constructions define mutually inverse
functors, proving \textup{(b)}.

For \textup{(c)}, let $B$ be a
$\U_{S,T}$-$\U_{R,P}$-bimodule.
Restricting the two actions gives an $S$-$R$-bimodule.
Define $p_B^S$ by the left action of the operator generator of
$\U_{S,T}$ and $p_B^R$ by the right action of the operator generator
of $\U_{R,P}$.
By \textup{(a)} and \textup{(b)}, $(B,p_B^S)$ is a left
$(S,T)$-module and $(B,p_B^R)$ is a right $(R,P)$-module.
Since the left and right actions commute,
$p_B^R(sb)=s p_B^R(b)$ and
$p_B^S(br)=p_B^S(b)r$ for all
$s\in S$, $r\in R$, and $b\in B$.
The actions of the two operator generators also commute, so
$p_B^S\circ p_B^R=p_B^R\circ p_B^S$.
Thus $(B,p_B^S,p_B^R)$ satisfies
Definition~\ref{def:bimodule}.

Conversely, let $(B,p_B^S,p_B^R)$ be an
$(S,T)$-$(R,P)$-bimodule.
Parts~\textup{(a)} and~\textup{(b)} give a left
$\U_{S,T}$-action and a right $\U_{R,P}$-action on $B$.
The $S$-$R$-bimodule identity and the compatibility conditions in
Definition~\ref{def:bimodule} show that these actions commute on the
algebra generators.
Hence they commute on the whole rings, making $B$ a
$\U_{S,T}$-$\U_{R,P}$-bimodule.

These constructions are mutually inverse.
An $S$-$R$-bimodule homomorphism commutes with both module operators
if and only if it is a $\U_{S,T}$-$\U_{R,P}$-bimodule homomorphism,
since the two rings are generated by their respective algebras and
operator generators.
This proves \textup{(c)}.
\end{proof}

\begin{corollary}\label{cor:free-enveloping}
For every set $X$, there is a canonical isomorphism of left
$(R,P)$-modules
$F_{R,P}(X)\cong\U_{R,P}^{(X)}$,
where $\U_{R,P}^{(X)}$ is the free left $\U_{R,P}$-module with
basis $X$, regarded as a left $(R,P)$-module via
Theorem~\ref{thm:category-equivalence}.
\end{corollary}

\begin{proof}
Under the equivalence in Theorem~\ref{thm:category-equivalence},
the free left $\U_{R,P}$-module $\U_{R,P}^{(X)}$ satisfies the
universal property of the free left $(R,P)$-module on $X$.
The result therefore follows from
Theorem~\ref{thm:free-reynolds}.
\end{proof}

\subsection{Idempotent operators fixing the identity}
\label{subsec:idempotent-normal}

Throughout this subsection, let $(R,P)$ be a Reynolds algebra of
weight $-1$ such that $P^2=P$ and $P(1_R)=1_R$.
Set $A:=\im P$ and $K:=\ker P$.

By Lemma~\ref{lem:classical-idempotent}, $P$ is an averaging operator
and $R=A\oplus K$. Moreover, $A$ is a subalgebra of $R$ with identity
$1_R$. The averaging identities give $P(ar)=aP(r)$ and
$P(ra)=P(r)a$ for all $a\in A$ and $r\in R$.
Thus $P:R\to A$ is an $A$-bimodule homomorphism.

\begin{lemma}\label{lem:idempotent-relations}
The ideal $I_{R,P}\subseteq\mathcal F_R$ defined in
Definition~\ref{def:enveloping} is generated by
$\iota_R(a)Q-Q\iota_R(a)$ for $a\in A$ and
$Q\iota_R(k)Q$ for $k\in K$.
Consequently, in $\U_{R,P}$,
$Q^i rQ^j=P(r)Q^{i+j}$ for all $r\in R$ and integers $i,j\geq1$.
\end{lemma}

\begin{proof}
For $a\in A$, we have $P(a)=a$. By
Eq.~(\ref{eq:enveloping-relation}), since the weight is $-1$, the
defining generator of $I_{R,P}$ corresponding to $a$ is
$$
 \iota_R(P(a))Q-Q\iota_R(P(a))
 -Q\iota_R(a-P(a))Q
 =\iota_R(a)Q-Q\iota_R(a).
$$
For $k\in K$, we have $P(k)=0$. Again by
Eq.~(\ref{eq:enveloping-relation}), the defining generator
corresponding to $k$ is $-Q\iota_R(k)Q$.
Therefore
$\iota_R(a)Q-Q\iota_R(a)\in I_{R,P}$ for all $a\in A$, and
$Q\iota_R(k)Q\in I_{R,P}$ for all $k\in K$.

Let $J$ be the two-sided ideal of $\mathcal F_R$ generated by these
two families. The preceding argument shows that $J\subseteq I_{R,P}$.

Conversely, let $r\in R$. Since $P^2=P$, we have
$P(r)\in A$ and $r-P(r)\in K$. Set $a=P(r)$ and $k=r-P(r)$. By
Eq.~(\ref{eq:enveloping-relation}), since the weight is $-1$, the
defining generator of $I_{R,P}$ corresponding to $r$ is
$$
 \iota_R(P(r))Q-Q\iota_R(P(r))
 -Q\iota_R(r-P(r))Q.
$$
Since $P(r)=a$ and $r-P(r)=k$, we have
$$
 \iota_R(P(r))Q-Q\iota_R(P(r))
 -Q\iota_R(r-P(r))Q
 =
 \iota_R(a)Q-Q\iota_R(a)-Q\iota_R(k)Q.
$$
The right-hand side belongs to $J$ by the definition of $J$.
Hence every defining generator of $I_{R,P}$ lies in $J$, and therefore
$I_{R,P}\subseteq J$. Consequently, $I_{R,P}=J$.

Passing to the quotient $\U_{R,P}=\mathcal F_R/I_{R,P}$, we obtain
$aQ=Qa$ for all $a\in A$ and $QkQ=0$ for all $k\in K$.
For any $r\in R$, write $r=P(r)+(r-P(r))$. Then
$$
 QrQ
 =QP(r)Q+Q(r-P(r))Q
 =QP(r)Q.
$$
Since $P(r)\in A$, we have $QP(r)=P(r)Q$, and hence
$QrQ=P(r)Q^2$. Therefore, for integers $i,j\geq1$,
$$
 Q^i rQ^j
 =Q^{i-1}(QrQ)Q^{j-1}
 =Q^{i-1}P(r)Q^{j+1}.
$$
Again using $QP(r)=P(r)Q$, we obtain $Q^i rQ^j=P(r)Q^{i+j}$,
as claimed.
\end{proof}

Assign degree zero to the elements of $R$ and degree one to $Q$.
By Lemma~\ref{lem:idempotent-relations}, the defining ideal
$I_{R,P}$ is homogeneous, so $\U_{R,P}$ inherits a grading by
nonnegative integers.

The same lemma suggests a concrete description of its homogeneous
components. Indeed, for $a\in A$ and $n\geq1$, the relation $aQ=Qa$
gives
$$
 raQ^ns=rQ^nas.
$$
Thus the degree-$n$ expressions $rQ^ns$ are naturally parametrized
by $R\otimes_A R$. Moreover, for $i,j\geq1$,
$$
 (rQ^is)(uQ^jv)
 =rQ^i(su)Q^jv
 =rP(su)Q^{i+j}v
$$
by Lemma~\ref{lem:idempotent-relations}. This motivates the following
graded algebra.

For each $n\geq1$, let $(R\otimes_A R)t^n$ be a copy of
$R\otimes_A R$ placed in degree $n$. Set
$$
 \mathcal E_P(R)
 :=R\oplus\bigoplus_{n\geq1}(R\otimes_A R)t^n.
$$
For $r,s\in R$, write $(r\otimes s)t^n$ for the element corresponding
to $r\otimes_A s$ in degree $n$.

The degree-zero part has the original multiplication of $R$.
For $x,r,s\in R$ and $n\geq1$, define
$x((r\otimes s)t^n)=(xr\otimes s)t^n$ and
$((r\otimes s)t^n)x=(r\otimes sx)t^n$.
For $r,s,u,v\in R$ and $i,j\geq1$, define
$$
 ((r\otimes s)t^i)((u\otimes v)t^j)
 :=(rP(su)\otimes v)t^{i+j}.
$$
Extend these products bilinearly.

\begin{theorem}\label{thm:idempotent-normal-form}
The above multiplication makes $\mathcal E_P(R)$ a unital graded
associative algebra. Moreover, there is a graded algebra isomorphism
$$
 \Phi:\mathcal E_P(R)\longrightarrow\U_{R,P}
$$
determined by $\Phi(x)=x$ for $x\in R$ and
$\Phi((r\otimes s)t^n)=rQ^ns$ for $r,s\in R$ and $n\geq1$.

Consequently, every element of $\U_{R,P}$ is uniquely determined by
a degree-zero component $x\in R$ and a finitely supported family
$(\xi_n)_{n\geq1}$ with $\xi_n\in R\otimes_A R$.
\end{theorem}

\begin{proof}
We first verify that the multiplication on $\mathcal E_P(R)$ is
well-defined. The products involving the degree-zero part are the
usual left and right $R$-actions on $R\otimes_A R$.

For the product of two positive-degree elements, it remains to check
compatibility with the tensor product over $A$.
Let $a\in A$ and $r,s,u,v\in R$. Since $P$ is left
$A$-linear,
$$
 raP(su)\otimes_A v
 =rP(asu)\otimes_A v.
$$
Hence replacing $ra\otimes_A s$ by $r\otimes_A as$ in the first
factor does not change the product.
Similarly, since $P$ is right $A$-linear,
$$
 rP(sua)\otimes_A v
 =rP(su)a\otimes_A v
 =rP(su)\otimes_A av.
$$
Thus replacing $ua\otimes_A v$ by $u\otimes_A av$ in the second
factor also leaves the product unchanged. Therefore the multiplication
is well-defined.

We next prove associativity. By bilinearity, it is enough to consider
homogeneous elements represented by pure tensors.
Associativity in degree zero follows from associativity of $R$, and
the cases involving two degree-zero factors follow from the
$R$-bimodule structure of $R\otimes_A R$.

Let $r,s,u,v,x\in R$ and $i,j\geq1$. Then
\begin{align*}
 \bigl(x((r\otimes s)t^i)\bigr)((u\otimes v)t^j)
 &=(xrP(su)\otimes v)t^{i+j}=x\bigl(((r\otimes s)t^i)((u\otimes v)t^j)\bigr),\\
 \bigl(((r\otimes s)t^i)((u\otimes v)t^j)\bigr)x
 &=(rP(su)\otimes vx)t^{i+j}=((r\otimes s)t^i)\bigl(((u\otimes v)t^j)x\bigr),\\
 \bigl(((r\otimes s)t^i)x\bigr)((u\otimes v)t^j)
 &=(rP(sxu)\otimes v)t^{i+j}=((r\otimes s)t^i)\bigl(x((u\otimes v)t^j)\bigr).
\end{align*}
Hence associativity holds whenever exactly one factor has degree zero.

For $w,z\in R$ and $\ell\geq1$, we have
\begin{align*}
 \bigl(((r\otimes s)t^i)((u\otimes v)t^j)\bigr)((w\otimes z)t^\ell)
 &=(rP(su)P(vw)\otimes z)t^{i+j+\ell},\\
 ((r\otimes s)t^i)\bigl(((u\otimes v)t^j)((w\otimes z)t^\ell)\bigr)
 &=(rP(suP(vw))\otimes z)t^{i+j+\ell}.
\end{align*}
Since $P(vw)\in A$ and $P$ is right $A$-linear,
$P(suP(vw))=P(su)P(vw)$.
Thus these two expressions are equal.
Therefore $\mathcal E_P(R)$ is associative.
It is graded by construction, and $1_R$ is its identity.

Set $q:=(1_R\otimes1_R)t$.
For $a\in A$, the relation
$a\otimes_A1_R=1_R\otimes_Aa$ gives
$$
 aq=(a\otimes1_R)t=(1_R\otimes a)t=qa.
$$
For $k\in K$, we have
$$
 qkq=(P(k)\otimes1_R)t^2=0.
$$
Thus the inclusion $R\to\mathcal E_P(R)$ together with $Q\mapsto q$
satisfies the relations in
Lemma~\ref{lem:idempotent-relations}.
Hence it induces a unique unital algebra homomorphism
$\Psi:\U_{R,P}\longrightarrow\mathcal E_P(R)$
such that $\Psi(r)=r$ for $r\in R$ and $\Psi(Q)=q$.

We now define $\Phi$. For $a\in A$,
Lemma~\ref{lem:idempotent-relations} gives $aQ=Qa$, and therefore $raQ^ns=rQ^nas$.
Hence $(r\otimes s)t^n\mapsto rQ^ns$ is well-defined on
$R\otimes_A R$. Together with the identity on $R$, these maps define
a linear map
$\Phi:\mathcal E_P(R)\to\U_{R,P}$.

It remains to check multiplicativity. The degree-zero products and
the left and right $R$-actions are clearly preserved. For positive
degrees, Lemma~\ref{lem:idempotent-relations} gives $rQ^isuQ^jv=rP(su)Q^{i+j}v$,
which is precisely the image under $\Phi$ of
$$
 ((r\otimes s)t^i)((u\otimes v)t^j)
 =(rP(su)\otimes v)t^{i+j}.
$$
Thus $\Phi$ is a unital algebra homomorphism.

Since $P(1_R)=1_R$, the multiplication in $\mathcal E_P(R)$ gives
$q^n=(1_R\otimes1_R)t^n$ for all $n\geq1$.
Consequently,
 $rq^ns=(r\otimes s)t^n$.
Hence $\Psi\circ\Phi$ fixes $R$ and every homogeneous element
$(r\otimes s)t^n$, so it is the identity on $\mathcal E_P(R)$.
On the other hand, $\Phi\circ\Psi$ fixes $R$ and $Q$, which generate
$\U_{R,P}$, and therefore is the identity on $\U_{R,P}$.
Thus $\Phi$ and $\Psi$ are mutually inverse graded algebra
isomorphisms.

The final assertion follows from the direct-sum decomposition of
$\mathcal E_P(R)$.
\end{proof}

\begin{example}\label{ex:group-normal-form}
Let $P=P_G$ be the finite-group averaging operator from
Example~\ref{ex:group-averaging-algebra}. Then
$A=\im P=R^G$. Hence Theorem~\ref{thm:idempotent-normal-form} gives
$$
 \U_{R,P}\cong
 R\oplus\bigoplus_{n\geq1}(R\otimes_{R^G}R)t^n
$$
as graded algebras, where, for $i,j\geq1$,
$$
 ((r\otimes s)t^i)((u\otimes v)t^j)
 =
 \left(
 rP_G(su)\otimes v
 \right)t^{i+j}
 =
 \frac{1}{|G|}
 \sum_{g\in G}(r\,g(su)\otimes v)t^{i+j}.
$$
\end{example}

\section{Tensor products and Morita theory}\label{sec:tensor}

In this section, we introduce tensor products over a Reynolds algebra,
establish the Hom-tensor adjunction, and obtain Reynolds versions of the
Eilenberg-Watts theorem and Morita theory, with an application to matrix
Reynolds algebras.

\subsection{Tensor products over a Reynolds algebra}
\label{subsec:tensor-product}

The following definition is adapted from the Rota-Baxter tensor
product in~\cite{QiaoGaoGuo2019}.

\begin{definition}\label{def:reynolds-tensor}
Let $(R,P)$ be a Reynolds algebra of weight $\lambda$, let
$(M,p_M)$ be a right $(R,P)$-module, and let $(N,p_N)$ be a left
$(R,P)$-module. The \textbf{tensor product of $M$ and $N$ over
$(R,P)$} is the quotient
$$
 M\otimes_{(R,P)}N:=(M\otimes N)/J,
$$
where $M\otimes N$ is the usual tensor product and $J$ is the
submodule generated by
$$
 mr\otimes n-m\otimes rn,
 \quad
 p_M(m)\otimes n-m\otimes p_N(n)
$$
for all $m\in M$, $n\in N$, and $r\in R$.
We write $m\otimes_{(R,P)}n$ for the image of $m\otimes n$ in
this quotient.
\end{definition}

This tensor product has the following universal property.
Let $G$ be a module and let $f:M\times N\to G$ be a
bilinear map satisfying $f(mr,n)=f(m,rn)$ and
$f(p_M(m),n)=f(m,p_N(n))$ for all $m\in M$, $n\in N$,
and $r\in R$.
Then there exists a unique linear map
$\overline f:M\otimes_{(R,P)}N\to G$ such that
$\overline f(m\otimes_{(R,P)}n)=f(m,n)$.

\begin{theorem}\label{thm:tensor-enveloping}
There is a natural isomorphism
$$
 M\otimes_{(R,P)}N\cong M\otimes_{\U_{R,P}}N
$$
determined by
$m\otimes_{(R,P)}n\mapsto m\otimes_{\U_{R,P}}n$.
\end{theorem}

\begin{proof}
By Theorem~\ref{thm:category-equivalence}, regard $M$ as a right
$\U_{R,P}$-module and $N$ as a left $\U_{R,P}$-module, with
$mQ=p_M(m)$ and $Qn=p_N(n)$.

Let $G$ be a module and let $f:M\times N\to G$ be a bilinear map
satisfying
$f(mr,n)=f(m,rn)$ and
$f(p_M(m),n)=f(m,p_N(n))$
for all $m\in M$, $n\in N$, and $r\in R$.
Set
$$
 S:=\{u\in\U_{R,P}\mid
 f(mu,n)=f(m,un)\text{ for all }m\in M,\ n\in N\}.
$$
The defining properties of $f$ imply that $R\subseteq S$ and
$Q\in S$. Moreover, $S$ is closed under addition and scalar
multiplication. If $u,v\in S$, then
$$
 f(muv,n)
 =f(mu,vn)
 =f(m,u(vn))
 =f(m,(uv)n),
$$
so $uv\in S$. Hence $S$ is a subalgebra of $\U_{R,P}$ containing
$R$ and $Q$. Since $\U_{R,P}$ is generated by $R$ and $Q$, we have
$S=\U_{R,P}$.

Thus every bilinear map satisfying the two relations in
Definition~\ref{def:reynolds-tensor} also satisfies
$f(mu,n)=f(m,un)$ for every $u\in\U_{R,P}$.
Conversely, if a bilinear map satisfies
$f(mu,n)=f(m,un)$ for every $u\in\U_{R,P}$, then taking $u\in R$
and $u=Q$ gives the two relations in
Definition~\ref{def:reynolds-tensor}.

Therefore the universal properties of
$M\otimes_{(R,P)}N$ and $M\otimes_{\U_{R,P}}N$ give an isomorphism
$$
 \theta_{M,N}:M\otimes_{(R,P)}N
 \longrightarrow M\otimes_{\U_{R,P}}N
$$
determined by
$$
 \theta_{M,N}(m\otimes_{(R,P)}n)
 =m\otimes_{\U_{R,P}}n.
$$

It remains to verify naturality. Let
$\alpha:M\to M'$ be a morphism of right $(R,P)$-modules and
$\beta:N\to N'$ a morphism of left $(R,P)$-modules.
For $m\in M$ and $n\in N$, we have
$$
 \theta_{M',N'}
 \bigl(\alpha(m)\otimes_{(R,P)}\beta(n)\bigr)
 =
 \alpha(m)\otimes_{\U_{R,P}}\beta(n),
$$
while
$$
 (\alpha\otimes\beta)
 \bigl(\theta_{M,N}(m\otimes_{(R,P)}n)\bigr)
 =
 \alpha(m)\otimes_{\U_{R,P}}\beta(n).
$$
Hence the corresponding square commutes. Since pure tensors generate
the tensor products, the isomorphisms $\theta_{M,N}$ are natural in
both $M$ and $N$.
\end{proof}

Consequently, for every right $(R,P)$-module $(M,p_M)$, the functor
$M\otimes_{(R,P)}-:{}_{(R,P)}\mathsf{Mod}\to\kfield\text{-}\mathsf{Mod}$
is additive and right exact. Similarly, for every left $(R,P)$-module
$(N,p_N)$, the functor
$-\otimes_{(R,P)}N:\mathsf{Mod}_{(R,P)}\to\kfield\text{-}\mathsf{Mod}$
is additive and right exact.

\subsection{Tensor functors and the Hom-tensor adjunction}
\label{subsec:adjunction}

Let $(S,T)$ and $(R,P)$ be Reynolds algebras, and let
$(B,p_B^S,p_B^R)$ be an $(S,T)$-$(R,P)$-bimodule.
We first record the tensor and Hom constructions associated with $B$.

\begin{proposition}\label{prop:change-scalars}
For every left $(R,P)$-module $(N,p_N)$, the tensor product
$B\otimes_{(R,P)}N$ is a left $(S,T)$-module with
$$
 s(b\otimes n)=(sb)\otimes n,
 \quad
 p(b\otimes n)=p_B^S(b)\otimes n.
$$
For fixed $B$, this defines a functor
$$
 B\otimes_{(R,P)}-:
 {}_{(R,P)}\mathsf{Mod}\longrightarrow{}_{(S,T)}\mathsf{Mod}.
$$
\end{proposition}

\begin{proof}
By Theorem~\ref{thm:category-equivalence}, $B$ is a
$\U_{S,T}$-$\U_{R,P}$-bimodule and $N$ is a left
$\U_{R,P}$-module. By Theorem~\ref{thm:tensor-enveloping},
$B\otimes_{(R,P)}N$ is naturally identified with
$B\otimes_{\U_{R,P}}N$.
The latter is a left $\U_{S,T}$-module via
$u(b\otimes n)=(ub)\otimes n$.
Restricting this action to $S$ and to the operator generator of
$\U_{S,T}$ gives the stated formulas.

If $f:N\to N'$ is a morphism of left $(R,P)$-modules, then
$\id_B\otimes f$ is $\U_{S,T}$-linear and hence a morphism of left
$(S,T)$-modules. Thus the construction is functorial.
\end{proof}

Let $(L,p_L)$ be a left $(S,T)$-module.

\begin{lemma}\label{lem:hom-reynolds-module}
The module $\Hom_{(S,T)}(B,L)$ is a left $(R,P)$-module with
$$
 (r\cdot f)(b)=f(br),
 \quad
 p(f)(b)=f(p_B^R(b))
$$
for $r\in R$, $b\in B$, and $f\in\Hom_{(S,T)}(B,L)$.
\end{lemma}

\begin{proof}
Under Theorem~\ref{thm:category-equivalence},
$\Hom_{(S,T)}(B,L)=\Hom_{\U_{S,T}}(B,L)$.
The usual left $\U_{R,P}$-action on this Hom module is
$$
 (u\cdot f)(b)=f(bu),
 \quad u\in\U_{R,P}.
$$
Restricting to $u=r\in R$ gives $(r\cdot f)(b)=f(br)$, while taking
$u=Q$ gives
$p(f)(b)=f(bQ)=f(p_B^R(b))$.
The result follows from Theorem~\ref{thm:category-equivalence}.
\end{proof}

\begin{theorem}\label{thm:hom-tensor}
For every left $(R,P)$-module $(N,p_N)$ and every left
$(S,T)$-module $(L,p_L)$, there is a natural isomorphism
$$
 \Hom_{(S,T)}\bigl(B\otimes_{(R,P)}N,L\bigr)
 \cong
 \Hom_{(R,P)}
 \bigl(N,\Hom_{(S,T)}(B,L)\bigr).
$$
\end{theorem}

\begin{proof}
By Theorems~\ref{thm:category-equivalence} and
\ref{thm:tensor-enveloping}, the left-hand side is naturally identified
with $\Hom_{\U_{S,T}}
 \bigl(B\otimes_{\U_{R,P}}N,L\bigr)$,
while Lemma~\ref{lem:hom-reynolds-module} identifies the right-hand
side with $\Hom_{\U_{R,P}}
 \bigl(N,\Hom_{\U_{S,T}}(B,L)\bigr)$.
The usual Hom-tensor adjunction for the
$\U_{S,T}$-$\U_{R,P}$-bimodule $B$ gives a natural isomorphism
between these two modules.

Explicitly, a morphism
$F:B\otimes_{(R,P)}N\to L$ is sent to
$\Phi(F):N\to\Hom_{(S,T)}(B,L)$ defined by $\Phi(F)(n)(b)=F(b\otimes n)$.
Conversely, a morphism
$G:N\to\Hom_{(S,T)}(B,L)$ determines
$\Psi(G):B\otimes_{(R,P)}N\to L$ by $\Psi(G)(b\otimes n)=G(n)(b)$.
These two assignments are mutually inverse and, being the usual
Hom-tensor adjunction under the above identifications, are natural in
$N$ and $L$.
\end{proof}

\subsection{Eilenberg-Watts and Morita theory}
\label{subsec:morita}

We first apply the classical Eilenberg-Watts theorem
\cite{Eilenberg1961,Watts1960} to Reynolds module categories.

\begin{theorem}\label{thm:reynolds-watts}
Let $(R,P)$ and $(S,T)$ be Reynolds algebras, and let
$F:{}_{(R,P)}\mathsf{Mod}\to{}_{(S,T)}\mathsf{Mod}$ be an additive
linear functor. Then the following conditions are equivalent:
\begin{enumerate}[label=\textup{(\roman*)}]
\item $F$ is right exact and preserves arbitrary direct sums.

\item there is an $(S,T)$-$(R,P)$-bimodule $B$ such that
$F$ is naturally isomorphic to $B\otimes_{(R,P)}-$.
\end{enumerate}
In this case, one may take $B=F(\U_{R,P})$, where $\U_{R,P}$ is
regarded as a left $(R,P)$-module via
Theorem~\ref{thm:category-equivalence}. The representing bimodule is
unique up to isomorphism.
\end{theorem}

\begin{proof}
By Theorem~\ref{thm:category-equivalence}, $F$ may be regarded as an
additive linear functor
$F:\U_{R,P}\text{-}\mathsf{Mod}\to \U_{S,T}\text{-}\mathsf{Mod}$.

Suppose that \textup{(i)} holds. By the classical
Eilenberg-Watts theorem, $B:=F(\U_{R,P})$ carries a natural
$\U_{S,T}$-$\U_{R,P}$-bimodule structure and there is a natural isomorphism $B\otimes_{\U_{R,P}}-\cong F$.
Explicitly, if $\rho_u:\U_{R,P}\to \U_{R,P}$ denotes right multiplication by
$u\in \U_{R,P}$, then the right $\U_{R,P}$-action on $B$ is given by
$bu=F(\rho_u)(b)$.
Since $F$ is linear, the left and right scalar actions on $B$
agree. Hence Theorem~\ref{thm:category-equivalence}\textup{(c)}
identifies $B$ with an $(S,T)$-$(R,P)$-bimodule.
By Theorem~\ref{thm:tensor-enveloping}, the above isomorphism becomes $B\otimes_{(R,P)}-\cong F$.
Thus \textup{(ii)} holds.

Conversely, if \textup{(ii)} holds, then
Theorem~\ref{thm:tensor-enveloping} identifies
$B\otimes_{(R,P)}-$ with the ordinary tensor functor
$B\otimes_{\U_{R,P}}-$, which is right exact and preserves arbitrary direct
sums. Hence \textup{(i)} holds.

The uniqueness of $B$ follows from the uniqueness statement in the
classical Eilenberg-Watts theorem.
\end{proof}

A left $(R,P)$-module $(G,p_G)$ is called a \textbf{generator} if
every left $(R,P)$-module is a quotient of a direct sum of copies of
$(G,p_G)$. It is \textbf{finitely generated} if it is a quotient of
a free left $(R,P)$-module on a finite set. A finitely generated
projective generator is called a \textbf{progenerator}, where
projectivity is defined in Definition~\ref{def:projective-injective}.
By Theorem~\ref{thm:category-equivalence} and
Corollary~\ref{cor:free-enveloping}, these notions agree with the
corresponding notions for modules over $\U_{R,P}$.

\begin{theorem}\label{thm:morita}
Let $(R,P)$ and $(S,T)$ be Reynolds algebras, and let
$(B,p_B^S,p_B^R)$ be an $(S,T)$-$(R,P)$-bimodule.
The functor
$$
 B\otimes_{(R,P)}-:
 {}_{(R,P)}\mathsf{Mod}\longrightarrow{}_{(S,T)}\mathsf{Mod}
$$
is an equivalence if and only if:
\begin{enumerate}[label=\textup{(\roman*)}]
\item $(B,p_B^S)$ is a progenerator as a left $(S,T)$-module.

\item the right $\U_{R,P}$-action on $B$ induces an algebra
isomorphism
$$
 \U_{R,P}^{\mathrm{op}}
 \longrightarrow
 \operatorname{End}_{(S,T)}(B),
 \quad
 u\longmapsto(b\mapsto bu).
$$
\end{enumerate}
When these conditions hold,
$\Hom_{(S,T)}(B,-)$ is a quasi-inverse.
\end{theorem}

\begin{proof}
By Theorem~\ref{thm:category-equivalence}\textup{(c)}, $B$ may be
regarded as a $\U_{S,T}$-$\U_{R,P}$-bimodule, and
Theorem~\ref{thm:tensor-enveloping} identifies
$B\otimes_{(R,P)}-$ with $B\otimes_{\U_{R,P}}-$.

By the classical Morita theorem \cite{Morita1958,Lam1999},
$B\otimes_{\U_{R,P}}-$ is an equivalence if and only if $B$ is a
progenerator as a left $\U_{S,T}$-module and the right action induces an
isomorphism $\U_{R,P}^{\mathrm{op}}\cong\operatorname{End}_{\U_{S,T}}(B)$.
Under Theorem~\ref{thm:category-equivalence}, these two conditions are
precisely \textup{(i)} and \textup{(ii)}.

When they hold, the classical quasi-inverse
$\Hom_{\U_{S,T}}(B,-)$ corresponds, by
Lemma~\ref{lem:hom-reynolds-module}, to
$\Hom_{(S,T)}(B,-)$.
\end{proof}

We next give a concrete compatibility between the ring of Reynolds
operators and matrix algebras.
For $n\geq1$, let
$P_n:M_n(R)\to M_n(R)$ be the entrywise operator, let $I_n$ be the
identity matrix, and let $Q_n$ denote the operator generator of
$\U_{M_n(R),P_n}$.

\begin{proposition}\label{prop:matrix-envelope}
Let $(R,P)$ be a Reynolds algebra of weight $-1$ such that
$P(1_R)=1_R$. For every $n\geq1$, there is a canonical unital
algebra isomorphism
$$
 \U_{M_n(R),P_n}\longrightarrow M_n(\U_{R,P})
$$
that is the identity on $M_n(R)$ and sends $Q_n$ to $QI_n$.
\end{proposition}

\begin{proof}
The entrywise operator $P_n$ is a Reynolds operator of weight $-1$.
Let $A=(a_{ij})\in M_n(R)$. The $(i,j)$-entry of
$$
 P_n(A)(QI_n)-(QI_n)P_n(A)
 -(QI_n)(A-P_n(A))(QI_n)
$$
is
$$
 P(a_{ij})Q-QP(a_{ij})
 -Q(a_{ij}-P(a_{ij}))Q,
$$
which is zero by Eq.~(\ref{eq:enveloping-compact}).
Hence Proposition~\ref{prop:enveloping-universal}\textup{(b)}
gives a unital algebra homomorphism
$$
 \Phi:\U_{M_n(R),P_n}\longrightarrow M_n(\U_{R,P})
$$
that is the identity on $M_n(R)$ and satisfies
$\Phi(Q_n)=QI_n$.

Let $E_{ij}$ be the standard matrix units in $M_n(R)$.
Since $P(1_R)=1_R$, we have $P_n(E_{ij})=E_{ij}$.
The defining relation of $\U_{M_n(R),P_n}$ corresponding to $E_{ij}$ gives $E_{ij}Q_n=Q_nE_{ij}$
for all $i,j$.

Set $q:=E_{11}Q_nE_{11}$, which belongs to the corner algebra
$E_{11}\U_{M_n(R),P_n}E_{11}$. This algebra has identity $E_{11}$,
and the map $R\to E_{11}\U_{M_n(R),P_n}E_{11}$ given by $r\mapsto rE_{11}$ is a unital
algebra homomorphism.
Applying the defining relation of $\U_{M_n(R),P_n}$ to $rE_{11}$ and multiplying
on the left and right by $E_{11}$ gives
$$
 P(r)E_{11}q-qP(r)E_{11}
 -q(r-P(r))E_{11}q=0.
$$
Thus Proposition~\ref{prop:enveloping-universal}\textup{(b)}
gives a unital algebra homomorphism
$$
 \psi:\U_{R,P}\longrightarrow E_{11}\U_{M_n(R),P_n}E_{11}
$$
satisfying $\psi(r)=rE_{11}$ and $\psi(Q)=q$.

Define
$$
 \Psi:M_n(\U_{R,P})\longrightarrow \U_{M_n(R),P_n},
 \quad
 \Psi((u_{ij}))
 =\sum_{i,j=1}^{n}E_{i1}\psi(u_{ij})E_{1j}.
$$
Since $\psi(u)\in E_{11}\U_{M_n(R),P_n}E_{11}$ and
$E_{1j}E_{k1}=\delta_{jk}E_{11}$, the map $\Psi$ is a unital algebra
homomorphism. It is the identity on $M_n(R)$.
Moreover,
$$
 \Psi(QI_n)
 =\sum_{i=1}^{n}E_{i1}qE_{1i}
 =\sum_{i=1}^{n}Q_nE_{ii}
 =Q_n.
$$
Thus $\Psi\circ\Phi$ fixes $M_n(R)$ and $Q_n$, while
$\Phi\circ\Psi$ fixes $M_n(R)$ and $QI_n$.
Hence $\Phi$ and $\Psi$ are mutually inverse.
\end{proof}

\begin{corollary}\label{cor:matrix-morita}
Under the assumptions of Proposition~\ref{prop:matrix-envelope},
the categories of left $(R,P)$-modules and left
$(M_n(R),P_n)$-modules are equivalent.

Explicitly, one equivalence sends $(M,p_M)$ to
$(M^n,p_M^{\oplus n})$, while a quasi-inverse sends
$(N,p_N)$ to $E_{11}N$ with the restricted operator
$p_N|_{E_{11}N}$.
\end{corollary}

\begin{proof}
By Proposition~\ref{prop:matrix-envelope}, $\U_{M_n(R),P_n}\cong M_n(\U_{R,P})$.
The ordinary Morita equivalence between $\U_{R,P}$ and
$M_n(\U_{R,P})$, together with
Theorem~\ref{thm:category-equivalence}, gives the asserted
equivalence of Reynolds module categories.

Under this equivalence, a left $\U_{R,P}$-module $M$ is sent to the
column module $M^n$, and $Q_n$ acts as $QI_n$, hence componentwise as
$p_M^{\oplus n}$.
Conversely, the corner functor sends a left
$M_n(\U_{R,P})$-module $N$ to $E_{11}N$.
Since $E_{11}$ commutes with $Q_n$, this submodule is stable under
the module operator, which restricts to
$p_N|_{E_{11}N}$.
\end{proof}

\section{Homological properties of Reynolds modules}\label{sec:homological}

We now study projective, injective and flat Reynolds modules and the
derived functors of Hom and tensor product. Through the ring of Reynolds
operators, these notions agree with the corresponding notions in ordinary
module theory, while the resulting $\Ext$ and $\Tor$ groups show how homological
properties depend on the Reynolds operator.

\subsection{Projective, injective and flat Reynolds modules}
\label{subsec:projective-injective}
\label{subsec:flat}

These notions are closely related to the corresponding notions for
modules over the ring of Reynolds operators.

By Theorem~\ref{thm:category-equivalence}, a sequence of morphisms
of left $(R,P)$-modules is exact if and only if its underlying
sequence of $R$-modules is exact.

\begin{definition}\label{def:projective-injective}
\begin{enumerate}[label=\textup{(\alph*)}]
\item A left $(R,P)$-module $(L,p_L)$ is called
\textbf{projective} if, for every surjective morphism
$\pi:(M,p_M)\to(N,p_N)$ and every morphism
$f:(L,p_L)\to(N,p_N)$, there exists a morphism
$\widetilde f:(L,p_L)\to(M,p_M)$ such that
$\pi\circ\widetilde f=f$.

\item A left $(R,P)$-module $(E,p_E)$ is called
\textbf{injective} if, for every injective morphism
$\iota:(M,p_M)\to(N,p_N)$ and every morphism
$f:(M,p_M)\to(E,p_E)$, there exists a morphism
$\widetilde f:(N,p_N)\to(E,p_E)$ such that
$\widetilde f\circ\iota=f$.
\end{enumerate}
\end{definition}

Equivalently, $(L,p_L)$ is projective if and only if
$\Hom_{(R,P)}(L,-)$ is exact, and $(E,p_E)$ is injective if and
only if $\Hom_{(R,P)}(-,E)$ is exact.
Projective and injective right $(R,P)$-modules are defined
analogously.

\begin{proposition}\label{prop:projective-injective-enveloping}
Under the equivalence of Theorem~\ref{thm:category-equivalence},
a left $(R,P)$-module is projective, respectively injective, if
and only if the corresponding left $\U_{R,P}$-module is projective,
respectively injective. The analogous statement holds for right
$(R,P)$-modules.
\end{proposition}

\begin{proof}
Theorem~\ref{thm:category-equivalence} identifies morphisms of
$(R,P)$-modules with $\U_{R,P}$-module homomorphisms and preserves
surjective and injective maps. Hence the lifting property for
projectivity and the extension property for injectivity hold in one
category if and only if they hold in the other. The same argument
applies to right modules.
\end{proof}

Consequently, the categories of left and right $(R,P)$-modules have
enough projective and injective objects, since the corresponding
module categories over $\U_{R,P}$ do.

\begin{example}\label{ex:projective-injective}
Let $I$ be an injective module. Equip
$\Hom_{\kfield}(\U_{R,P},I)$ with the left $\U_{R,P}$-action
$$
 (u\cdot f)(v)=f(vu),
 \quad u,v\in\U_{R,P}.
$$
Then $\Hom_{\kfield}(\U_{R,P},I)$ is an injective left
$\U_{R,P}$-module.

Indeed, for every left $\U_{R,P}$-module $M$, there is a natural
isomorphism
$$
 \Hom_{\U_{R,P}}
 \bigl(M,\Hom_{\kfield}(\U_{R,P},I)\bigr)
 \cong
 \Hom_{\kfield}(M,I).
$$
Since $I$ is injective, the functor on the right is exact.
Hence $\Hom_{\kfield}(\U_{R,P},I)$ is injective as a left
$\U_{R,P}$-module.

Restricting the action to $R$ and letting $Q$ act by
$p(f)(v)=f(vQ)$ gives an injective left $(R,P)$-module by
Proposition~\ref{prop:projective-injective-enveloping}.

In particular, if $\kfield$ is a field, taking $I=\kfield$ gives
the injective left $(R,P)$-module
$\Hom_{\kfield}(\U_{R,P},\kfield)$.
\end{example}

We next consider flat Reynolds modules.

\begin{definition}\label{def:flat}
A left $(R,P)$-module $(N,p_N)$ is called \textbf{flat} if the functor
$$
 -\otimes_{(R,P)}N:
 \mathsf{Mod}_{(R,P)}\longrightarrow\kfield\text{-}\mathsf{Mod}
$$
is exact. A right $(R,P)$-module is called \textbf{flat} analogously.
\end{definition}

\begin{theorem}\label{thm:flat-equivalence}
A left $(R,P)$-module is flat if and only if the corresponding left
$\U_{R,P}$-module is flat. The analogous statement holds for right
$(R,P)$-modules.
\end{theorem}

\begin{proof}
Let $(N,p_N)$ be a left $(R,P)$-module.
By Theorem~\ref{thm:category-equivalence}, exact sequences of right
$(R,P)$-modules correspond to exact sequences of right
$\U_{R,P}$-modules. Moreover,
Theorem~\ref{thm:tensor-enveloping} gives a natural isomorphism
$$
 M\otimes_{(R,P)}N
 \cong
 M\otimes_{\U_{R,P}}N
$$
for every right $(R,P)$-module $M$, compatible with morphisms in $M$.
Hence $-\otimes_{(R,P)}N$ is exact if and only if
$-\otimes_{\U_{R,P}}N$ is exact.
This proves the assertion for left modules.
The right-module case is analogous.
\end{proof}

In particular, every projective left or right $(R,P)$-module is flat.
Since free left $(R,P)$-modules correspond to free left modules over
$\U_{R,P}$, every free left $(R,P)$-module is flat. The analogous
statement holds for right modules.

\subsection{Derived functors}\label{subsec:derived}

We now define the derived functors of Hom and tensor products in the
categories of Reynolds modules, using the standard conventions of
homological algebra~\cite[Chapters~2-3]{Weibel1994}.

\begin{definition}\label{def:derived-functors}
Let $(M,p_M)$ and $(N,p_N)$ be left $(R,P)$-modules, and let
$(L,p_L)$ be a right $(R,P)$-module. For $n\geq0$, define
$\Ext^n_{(R,P)}(M,N)$ as the $n$-th right derived functor of
$\Hom_{(R,P)}(M,-)$ evaluated at $N$, and define
$\Tor_n^{(R,P)}(L,M)$ as the $n$-th left derived functor of
$L\otimes_{(R,P)}-$ evaluated at $M$.
\end{definition}

\begin{proposition}\label{prop:derived-properties}
For every $n\geq0$, there are natural isomorphisms
$$
 \Ext^n_{(R,P)}(M,N)
 \cong
 \Ext^n_{\U_{R,P}}(M,N),
 \quad
 \Tor_n^{(R,P)}(L,M)
 \cong
 \Tor_n^{\U_{R,P}}(L,M).
$$
In particular,
$\Ext^0_{(R,P)}(M,N)\cong\Hom_{(R,P)}(M,N)$ and
$\Tor_0^{(R,P)}(L,M)\cong L\otimes_{(R,P)}M$.
Projective or injective resolutions compute $\Ext$, while projective
or flat resolutions compute $\Tor$. The usual vanishing properties
and long exact sequences also hold.
\end{proposition}

\begin{proof}
By Theorem~\ref{thm:category-equivalence}, the categories of left and
right $(R,P)$-modules are equivalent to the corresponding module
categories over $\U_{R,P}$.

By Proposition~\ref{prop:projective-injective-enveloping},
projective and injective $(R,P)$-modules correspond precisely to
projective and injective $\U_{R,P}$-modules. Moreover,
$\Hom_{(R,P)}(M,N)=\Hom_{\U_{R,P}}(M,N)$.
Thus the complexes computing the corresponding Ext groups agree,
which gives
$\Ext^n_{(R,P)}(M,N)\cong\Ext^n_{\U_{R,P}}(M,N)$.

Similarly, Theorem~\ref{thm:flat-equivalence} identifies flat
$(R,P)$-modules with flat $\U_{R,P}$-modules, while
Theorem~\ref{thm:tensor-enveloping} gives
$L\otimes_{(R,P)}M\cong L\otimes_{\U_{R,P}}M$ naturally.
Hence the complexes computing the corresponding Tor groups agree.
The remaining assertions follow from the standard properties of
Ext and Tor over a ring.
\end{proof}

The composition of tensor functors gives the following spectral
sequence. see~\cite[Section~5.8]{Weibel1994}.

\begin{theorem}\label{thm:tor-spectral}
Let $(B,p_B^S,p_B^R)$ be an $(S,T)$-$(R,P)$-bimodule and suppose
that $B$ is flat as a left $(S,T)$-module. Let $(C,p_C)$ be a right
$(S,T)$-module and let $(N,p_N)$ be a left $(R,P)$-module.
Then $\Tor_q^{(R,P)}(B,N)$ carries a natural left $(S,T)$-module
structure, while $C\otimes_{(S,T)}B$ carries a natural right
$(R,P)$-module structure. Moreover, there is a natural
first-quadrant spectral sequence
\begin{equation}\label{eq:tor-spectral}
 E^2_{p,q}
 =
 \Tor_p^{(S,T)}
 \bigl(C,\Tor_q^{(R,P)}(B,N)\bigr)
 \Longrightarrow
 \Tor_{p+q}^{(R,P)}
 \bigl(C\otimes_{(S,T)}B,N\bigr).
\end{equation}
\end{theorem}

\begin{proof}
By Theorem~\ref{thm:category-equivalence}, $B$ is a
$\U_{S,T}$-$\U_{R,P}$-bimodule, $C$ is a right
$\U_{S,T}$-module, and $N$ is a left $\U_{R,P}$-module.
Consider the right exact functors
$$
 B\otimes_{\U_{R,P}}-:
 \U_{R,P}\text{-}\mathsf{Mod}\longrightarrow
 \U_{S,T}\text{-}\mathsf{Mod}
$$
and
$$
 C\otimes_{\U_{S,T}}-:
 \U_{S,T}\text{-}\mathsf{Mod}\longrightarrow
 \kfield\text{-}\mathsf{Mod}.
$$
Let $E$ be a projective left $\U_{R,P}$-module. It is a direct
summand of a free module $\U_{R,P}^{(I)}$ for some set $I$.
Hence $B\otimes_{\U_{R,P}}E$ is a direct summand of $B^{(I)}$.
By Theorem~\ref{thm:flat-equivalence}, the assumed left flatness
of $B$ is its flatness as a left $\U_{S,T}$-module. Thus
$B\otimes_{\U_{R,P}}E$ is flat over $\U_{S,T}$ and is acyclic
for the left derived functors of $C\otimes_{\U_{S,T}}-$.

The homological Grothendieck spectral sequence therefore gives
$$
 E^2_{p,q}
 =\Tor_p^{\U_{S,T}}
 \bigl(C,\Tor_q^{\U_{R,P}}(B,N)\bigr)
 \Longrightarrow
 \Tor_{p+q}^{\U_{R,P}}
 \bigl(C\otimes_{\U_{S,T}}B,N\bigr).
$$
The left $\U_{S,T}$-action on $B$ induces the left
$\U_{S,T}$-module structure on $\Tor_q^{\U_{R,P}}(B,N)$.
The right $\U_{R,P}$-action on $B$ induces the right
$\U_{R,P}$-module structure on $C\otimes_{\U_{S,T}}B$.
Theorems~\ref{thm:category-equivalence} and
\ref{thm:tensor-enveloping}, together with
Proposition~\ref{prop:derived-properties}, translate this spectral
sequence into Eq.~(\ref{eq:tor-spectral}).
\end{proof}

\begin{corollary}\label{cor:derived-tensor-associativity}
Under the assumptions of Theorem~\ref{thm:tor-spectral}, suppose in
addition that $B$ is flat as a right $(R,P)$-module.
Then, for every $n\geq0$, there is a natural isomorphism
$$
 \Tor_n^{(S,T)}
 \bigl(C,B\otimes_{(R,P)}N\bigr)
 \cong
 \Tor_n^{(R,P)}
 \bigl(C\otimes_{(S,T)}B,N\bigr).
$$
For $n=0$, this is the usual associativity isomorphism of tensor
products.
\end{corollary}

\begin{proof}
Since $B$ is flat as a right $(R,P)$-module,
$\Tor_q^{(R,P)}(B,N)=0$ for $q>0$, while
$\Tor_0^{(R,P)}(B,N)=B\otimes_{(R,P)}N$.
Hence the spectral sequence in
Theorem~\ref{thm:tor-spectral} is concentrated in the row $q=0$
and degenerates at $E^2$, giving the stated isomorphisms.
\end{proof}

The Hom-tensor adjunction has the following derived form.

\begin{theorem}\label{thm:ext-spectral}
Let $(B,p_B^S,p_B^R)$ be an $(S,T)$-$(R,P)$-bimodule and suppose
that $B$ is flat as a right $(R,P)$-module. Let $(N,p_N)$ be a left
$(R,P)$-module and let $(L,p_L)$ be a left $(S,T)$-module.
Then $\Ext^q_{(S,T)}(B,L)$ carries a natural left $(R,P)$-module
structure for every $q\geq0$, and there is a natural cohomological
first-quadrant spectral sequence
\begin{equation}\label{eq:ext-spectral}
 E_2^{p,q}
 =
 \Ext^p_{(R,P)}
 \bigl(N,\Ext^q_{(S,T)}(B,L)\bigr)
 \Longrightarrow
 \Ext^{p+q}_{(S,T)}
 \bigl(B\otimes_{(R,P)}N,L\bigr).
\end{equation}
\end{theorem}

\begin{proof}
Choose an injective resolution $L\to I^\bullet$ in
${}_{(S,T)}\mathsf{Mod}$.
By Lemma~\ref{lem:hom-reynolds-module},
each $\Hom_{(S,T)}(B,I^j)$ is a left $(R,P)$-module, and the
differentials are morphisms of left $(R,P)$-modules.
Hence the cohomology groups
$\Ext^q_{(S,T)}(B,L)$ inherit natural left $(R,P)$-module
structures.

We verify the acyclicity condition for the Grothendieck spectral
sequence. Let $I$ be an injective left $(S,T)$-module.
Since $B$ is flat as a right $(R,P)$-module, every monomorphism
$X\to Y$ of left $(R,P)$-modules induces a monomorphism
$$
 B\otimes_{(R,P)}X
 \longrightarrow
 B\otimes_{(R,P)}Y.
$$
By Theorem~\ref{thm:hom-tensor}, applying
$\Hom_{(R,P)}(-,\Hom_{(S,T)}(B,I))$ to $X\to Y$ is naturally
identified with applying $\Hom_{(S,T)}(-,I)$ to this induced
monomorphism. The latter map is surjective because $I$ is injective.
Thus $\Hom_{(S,T)}(B,I)$ is injective as a left $(R,P)$-module.

Applying the Grothendieck spectral sequence to the left exact
functors
$$
 \Hom_{(S,T)}(B,-)
 \quad\text{and}\quad
 \Hom_{(R,P)}(N,-)
$$
gives
$$
 E_2^{p,q}
 =
 \Ext^p_{(R,P)}
 \bigl(N,\Ext^q_{(S,T)}(B,L)\bigr).
$$
By Theorem~\ref{thm:hom-tensor}, their composite is naturally
isomorphic to
$\Hom_{(S,T)}(B\otimes_{(R,P)}N,-)$.
Hence the spectral sequence converges to
$\Ext^{p+q}_{(S,T)}(B\otimes_{(R,P)}N,L)$, proving
Eq.~(\ref{eq:ext-spectral}).
\end{proof}

\begin{corollary}\label{cor:derived-adjunction}
Under the assumptions of Theorem~\ref{thm:ext-spectral}, suppose in
addition that $B$ is projective as a left $(S,T)$-module.
Then, for every $n\geq0$, there is a natural isomorphism
$$
 \Ext^n_{(S,T)}
 \bigl(B\otimes_{(R,P)}N,L\bigr)
 \cong
 \Ext^n_{(R,P)}
 \bigl(N,\Hom_{(S,T)}(B,L)\bigr).
$$
For $n=0$, this is the Hom-tensor adjunction in
Theorem~\ref{thm:hom-tensor}.
\end{corollary}

\begin{proof}
Since $B$ is projective as a left $(S,T)$-module,
$\Ext^q_{(S,T)}(B,L)=0$ for $q>0$, while
$\Ext^0_{(S,T)}(B,L)=\Hom_{(S,T)}(B,L)$.
Hence the spectral sequence in Theorem~\ref{thm:ext-spectral} is
concentrated in the row $q=0$ and degenerates at $E_2$.
The stated isomorphisms follow.
\end{proof}

The following computation shows that even for the fixed underlying
algebra $\kfield$, different Reynolds operators can lead to sharply
different homological behavior.

\begin{proposition}\label{prop:one-dimensional-derived}
Let $\kfield$ be a field.
\begin{enumerate}[label=\textup{(\alph*)}]
\item Consider the Reynolds algebra $(\kfield,\id_{\kfield})$ of
weight $-1$. For $a\in\kfield$, let $\kfield_a$ denote the
one-dimensional left or right $(\kfield,\id_{\kfield})$-module on
which the module operator acts as multiplication by $a$.
Then, for $a,b\in\kfield$ and $n\geq0$,
$$
 \Ext^n_{(\kfield,\id)}(\kfield_a,\kfield_b)
 \cong
 \Tor_n^{(\kfield,\id)}(\kfield_a,\kfield_b)
 \cong
 \begin{cases}
  \kfield,& a=b\text{ and }n\in\{0,1\},\\
  0,&\text{otherwise}.
 \end{cases}
$$
For the Tor group, $\kfield_a$ is regarded as a right module and
$\kfield_b$ as a left module. In particular,
$\U_{\kfield,\id}$ has global dimension $1$.

\item Consider the Reynolds algebra $(\kfield,0)$ of weight $-1$, and let
$\kfield_0$ denote the one-dimensional left or right
$(\kfield,0)$-module with zero module operator.
Then, for every $n\geq0$,
$$
 \Ext^n_{(\kfield,0)}(\kfield_0,\kfield_0)\cong\kfield,
 \quad
 \Tor_n^{(\kfield,0)}(\kfield_0,\kfield_0)\cong\kfield.
$$
In particular, $\U_{\kfield,0}$ has infinite global dimension.
\end{enumerate}
\end{proposition}

\begin{proof}
For \textup{(a)}, by the defining relation of
$\U_{\kfield,\id}$, we have
$\U_{\kfield,\id}\cong\kfield[Q]$.
Under Theorem~\ref{thm:category-equivalence},
$\kfield_a$ corresponds to $\kfield[Q]/(Q-a)$.
Since multiplication by $Q-a$ is injective, there is a free
resolution
$$
 0\longrightarrow\kfield[Q]
 \xrightarrow{\,Q-a\,}\kfield[Q]
 \longrightarrow\kfield_a\longrightarrow0.
$$
Applying $\Hom_{\kfield[Q]}(-,\kfield_b)$ gives a two-term complex
whose differential is multiplication by $b-a$.
If $a=b$, this differential is zero. If $a\neq b$, it is an
isomorphism. This gives the stated Ext groups.

For Tor, use the analogous free resolution of the left module
$\kfield_b$ and tensor it with the right module $\kfield_a$.
The resulting differential is multiplication by $a-b$, giving the
stated Tor groups. Finally, $\kfield[Q]$ has global dimension $1$.

For \textup{(b)}, the defining relation gives
$Q^2=0$, and hence
$\U_{\kfield,0}\cong\kfield[Q]/(Q^2)$.
The module $\kfield_0$ corresponds to the quotient by $(Q)$.
Consider the free resolution
$$
 \cdots\xrightarrow{\,Q\,}\U_{\kfield,0}
 \xrightarrow{\,Q\,}\U_{\kfield,0}
 \xrightarrow{\,Q\,}\U_{\kfield,0}
 \longrightarrow\kfield_0\longrightarrow0.
$$
Since the kernel and image of multiplication by $Q$ are both $(Q)$,
this sequence is exact.

Applying $\Hom_{\U_{\kfield,0}}(-,\kfield_0)$ gives one copy of
$\kfield$ in every degree with zero differentials. Tensoring the
same resolution with the right module $\kfield_0$ gives the same
result for Tor. Thus the stated Ext and Tor groups are isomorphic to
$\kfield$ in every degree, so $\kfield_0$ has infinite projective
dimension. Therefore $\U_{\kfield,0}$ has infinite global dimension.
\end{proof}

\smallskip
\noindent\textbf{Acknowledgments.}
This work was supported by the National Natural Science Foundation of
China (Grant Nos.~12671029, 12301025 and 12101316). The authors thank
Professor Li Guo for helpful discussions.

\smallskip
\noindent\textbf{Declaration of generative AI use.}
During the preparation of this work, the authors used ChatGPT (OpenAI)
to assist with language editing, checking the consistency of mathematical
notation and references, and LaTeX formatting. The authors reviewed and
edited the output as needed and take full responsibility for the content
of the article.

\smallskip
\noindent\textbf{Declaration of interests.}
The authors declare that they have no conflicts of interest.

\smallskip
\noindent\textbf{Data availability.}
No datasets were generated or analyzed in this work.

\enlargethispage{2\baselineskip}


\begin{thebibliography}{99}\sloppy

\bibitem{Atkinson1963}
F. V. Atkinson,
Some aspects of Baxter's functional equation,
\emph{J. Math. Anal. Appl.} \textbf{7} (1963), 1-30.

\bibitem{Baxter1960}
G. Baxter,
An analytic problem whose solution follows from a simple algebraic identity,
\emph{Pacific J. Math.} \textbf{10} (1960), 731-742.

\bibitem{BillikRota1960}
M. Billik and G. C. Rota,
On Reynolds operators in finite-dimensional algebras,
\emph{J. Math. Mech.} \textbf{9} (1960), 927-932.

\bibitem{Birkhoff1961}
G. Birkhoff,
Lattices in applied mathematics. II. Averaging operators,
in: \emph{Lattice Theory}, Proc. Sympos. Pure Math., Vol. 2,
Amer. Math. Soc., Providence, RI, 1961, pp. 163-184.


\bibitem{Das2021Lie}
A. Das,
Twisted Rota-Baxter operators and Reynolds operators on Lie algebras and NS-Lie algebras,
\emph{J. Math. Phys.} \textbf{62} (2021), no.~9, 091701, 23 pp.

\bibitem{DerksenKemper2015}
H. Derksen and G. Kemper,
\emph{Computational Invariant Theory},
second enlarged ed., Encyclopaedia of Mathematical Sciences, Vol. 130,
Springer, Heidelberg, 2015.

\bibitem{Dubreil1957}
M. L. Dubreil-Jacotin,
\emph{\'Etude alg\'ebrique des transformations de Reynolds},
in: Colloque d'alg\`ebre sup\'erieure, Bruxelles, 1956, pp. 9-27, 1957.

\bibitem{Eilenberg1961}
S. Eilenberg,
Abstract description of some basic functors,
\emph{J. Indian Math. Soc. (N.S.)} \textbf{24} (1960), 231-234 (1961).


\bibitem{Guo2012}
L. Guo,
\emph{An Introduction to Rota-Baxter Algebra},
Surveys of Modern Mathematics, Vol. 4,
International Press, Somerville, MA, 2012.

\bibitem{GuoGustavsonLi2024}
L. Guo, R. Gustavson and Y. Li,
Generalized Reynolds algebras from Volterra integrals and their free construction by complete shuffle product,
arXiv:2411.02633 (2024).

\bibitem{GuoLin2021}
L. Guo and Z. Lin,
Representations and modules of Rota-Baxter algebras,
\emph{Asian J. Math.} \textbf{25} (2021), no.~6, 841-870.

\bibitem{GuoZhang2023}
S. J. Guo and Y. Zhang,
The cohomology of relative cocycle weighted Reynolds operators and NS-pre-Lie algebras,
\emph{Comm. Algebra} \textbf{51} (2023), no.~12, 5313-5331.

\bibitem{HPZ}
J. He, X. S. Peng and Y. Zhang,
Multiple Rota-Baxter algebra and multiple Rota-Baxter modules,
arXiv:2504.16643 (2025).


\bibitem{HouSheng2023}
S. Hou and Y. Sheng,
Cohomologies of Reynolds operators on $n$-Lie algebras,
\emph{Comm. Algebra} \textbf{51} (2023), no.~2, 779-798.


\bibitem{Lam1999}
T. Y. Lam,
\emph{Lectures on Modules and Rings},
Graduate Texts in Mathematics, Vol. 189,
Springer-Verlag, New York, 1999.


\bibitem{Maithani2026}
A. Maithani,
Splitting the difference: Computations of the Reynolds operator in classical invariant theory,
\emph{Bull. Lond. Math. Soc.} \textbf{58} (2026), no.~1, e70175.

\bibitem{Miller1966}
J. B. Miller,
Averaging and Reynolds operators on Banach algebras. I. Representation by derivations and antiderivations,
\emph{J. Math. Anal. Appl.} \textbf{14} (1966), 527-548.

\bibitem{Morita1958}
K. Morita,
Duality for modules and its applications to the theory of rings with minimum condition,
\emph{Sci. Rep. Tokyo Kyoiku Daigaku Sect. A} \textbf{6} (1958), no.~150, 83-142.

\bibitem{MumfordFogartyKirwan1994}
D. Mumford, J. Fogarty and F. Kirwan,
\emph{Geometric Invariant Theory},
third enlarged ed., Ergebnisse der Mathematik und ihrer Grenzgebiete (2), Vol. 34,
Springer-Verlag, Berlin, 1994.

\bibitem{PeiGuo2015}
J. Pei and L. Guo,
Averaging algebras, Schr\"oder numbers, rooted trees and operads,
\emph{J. Algebraic Combin.} \textbf{42} (2015), 73-109.


\bibitem{QiaoGaoGuo2019}
L. Qiao, X. Gao and L. Guo,
Rota-Baxter modules toward derived functors,
\emph{Algebr. Represent. Theory} \textbf{22} (2019), no.~2, 321-343.

\bibitem{QiaoPei2026}
L. Qiao and J. Pei,
Differential modules toward derived functors,
\emph{Front. Math.} \textbf{21} (2026), no.~3, 599-622.

\bibitem{Reynolds1895}
O. Reynolds,
On the dynamical theory of incompressible viscous fluids and the determination of the criterion,
\emph{Philos. Trans. Roy. Soc. London Ser. A} \textbf{186} (1895), 123-164.

\bibitem{Rota1964}
G. C. Rota,
Reynolds operators,
in: \emph{Proc. Sympos. Appl. Math.}, Vol. 16,
Amer. Math. Soc., Providence, RI, 1964, pp. 70-83.


\bibitem{Watts1960}
C. E. Watts,
Intrinsic characterizations of some additive functors,
\emph{Proc. Amer. Math. Soc.} \textbf{11} (1960), 5-8.

\bibitem{Weibel1994}
C. A. Weibel,
\emph{An Introduction to Homological Algebra},
Cambridge Studies in Advanced Mathematics, Vol. 38,
Cambridge University Press, Cambridge, 1994.

\bibitem{ZhangGaoGuo2021}
T. Zhang, X. Gao and L. Guo,
Reynolds algebras and their free objects from bracketed words and rooted trees,
\emph{J. Pure Appl. Algebra} \textbf{225} (2021), no.~12, 106766.

\bibitem{ZhangGaoGuo2023}
T. Zhang, X. Gao and L. Guo,
Construction of free commutative Reynolds algebras by Gr\"obner-Shirshov bases,
\emph{J. Symbolic Comput.} \textbf{119} (2023), 64-80.

\end{thebibliography}
\end{document}